%% file: double_dual.tex
\documentclass[twoside, 11pt]{amsart} 

\usepackage{preamble} 

\usepackage{fancyhdr}
\usepackage{geometry}
\usepackage[normalem]{ulem}
\usepackage{parskip,xfrac} 
\usepackage{tikz, tikz-cd}
        \usetikzlibrary{trees}
        \usetikzlibrary{external}
\newcommand{\EK}{\mathbb{S}^3_K}
\newcommand{\ts}{\mathbb{S}^3}
\title{Dehn surgery functions are never injective}
\author{Kyle Hayden, Lisa Piccirillo, and Laura Wakelin}
\address{}
\email{}
\date{} 

\definecolor{red}{rgb}{.65,.15,.15}
\definecolor{blue}{rgb}{.15,.15,.75}
\definecolor{green}{rgb}{.15,.75,.15}
\definecolor{color0}{rgb}{0.14,0.2,0.5} 

\definecolor{editred}{rgb}{.95,.25,.25}

\begin{document} 

\pagestyle{fancy}
\fancyhead{}
\fancyhead[CE]{\scriptsize\uppercase{Kyle Hayden, Lisa Piccirillo, and Laura Wakelin}} 
\fancyhead[CO]{\scriptsize\uppercase{Dehn surgery functions are never injective}} 
\fancyfoot[C]{\scriptsize\thepage} 

\begin{abstract} 
We prove that, for each fixed rational number $p/q \in \mathbb{Q}$, there exists a pair of distinct knots whose $p/q$-surgeries are orientation-preservingly homeomorphic. 
This confirms a 1978 conjecture of Gordon.
\end{abstract} 

\maketitle

\section{Introduction}

Dehn surgery is a construction that assigns closed, oriented 3-manifolds to links in $\mathbb{S}^3$. 
This construction is surjective in the sense that every such 3-manifold can be obtained in this way \cite{Lickorish,Wallace}; this foundational result  underpins the deep connection between knot theory and 3-manifold topology. 
However, it is highly non-injective: every 3-manifold admits many distinct link-surgery descriptions. 

If one restricts attention to knots, however, then surjectivity and injectivity become more subtle questions. 
For each fixed slope $p/q \in \mathbb{Q}$, consider the $p/q$-framed Dehn surgery function $K \mapsto \mathbb{S}^3_K(p/q)$, regarded as a map from the set of knots in $\mathbb{S}^3$ to the set of closed, oriented 3-manifolds with $H_1 \cong \mathbb{Z}/p\mathbb{Z}$.
In \cite[§15]{Gordon}, Gordon conjectured that this map is neither surjective nor injective for each fixed $p/q\in\mathbb{Q}$. 
The non-surjectivity part of this conjecture was established by Gordon-Luecke \cite{GL}.\footnote{For example, a reducible $\mathbb{Z}/p\mathbb{Z}$-homology sphere with no lens space summand is never realised as surgery on a knot.} 
In this paper, we complete the proof of the non-injectivity part of the conjecture. 

\begin{theorem*}
\label{theorem:non-characterising}
    For every $p/q \in \mathbb{Q}$, there exist knots $K \neq K'$ such that $\hspace{1.4pt} \mathbb{S}^3_K(p/q) \cong \mathbb{S}^3_{K'}(p/q)$. 
\end{theorem*}

This result was previously known on a subset of the rationals. 
Lickorish gave the first proof that knot-surgery descriptions are not unique by constructing examples of distinct knots that share a $\pm1$-surgery \cite{Lickorish:surgery}. 
Lickorish's construction extends to all slopes $\pm1/q$, as described by Brakes \cite{Brakes} and exhibited explicitly in e.g.~\cite{McCoy:pretzel, Wakelin}. 
Analogous results for all integral surgeries were obtained in  \cite{Akbulut, Brakes, Akbulut:shake}. 
Non-injectivity was also shown for some slopes $p/q$ with $p \equiv 1 \mod q$  by Brakes; Kawauchi extended this to all such slopes \cite{Kawauchi}.

\subsection{Constructing and distinguishing the knots} 

For any rational $p/q$, we use a modification of the RBG link construction (cf.~\cite{Piccirillo:shake}) to present two knots $K_B$ and $K_G$ with orientation-preservingly homeomorphic $p/q$-surgeries; see Section~\ref{section:construction}. 
We describe these knots explicitly; Figure \ref{figure:K_P_simple} shows a template used to depict $K_B$ and $K_G$, where $P$ represents a pair of satellite patterns. 
Compared to previous constructions (e.g.~\cite{Lickorish:surgery, Akbulut, Brakes, Akbulut:exotic, GM:ribbon, Kawauchi, Osoinach, Yasui, BM}), many of which can be converted into an RBG framework, our approach places a more explicit emphasis on the role of the \emph{dual knot}: the knot $K^*$ in $\mathbb{S}^3_{K}(p/q)$ given by the core of the surgery solid torus. 
This ultimately plays a key role in avoiding constraints on the slope $p/q$. 

\begin{figure}
    \centering 
    \def\svgwidth{0.92\linewidth}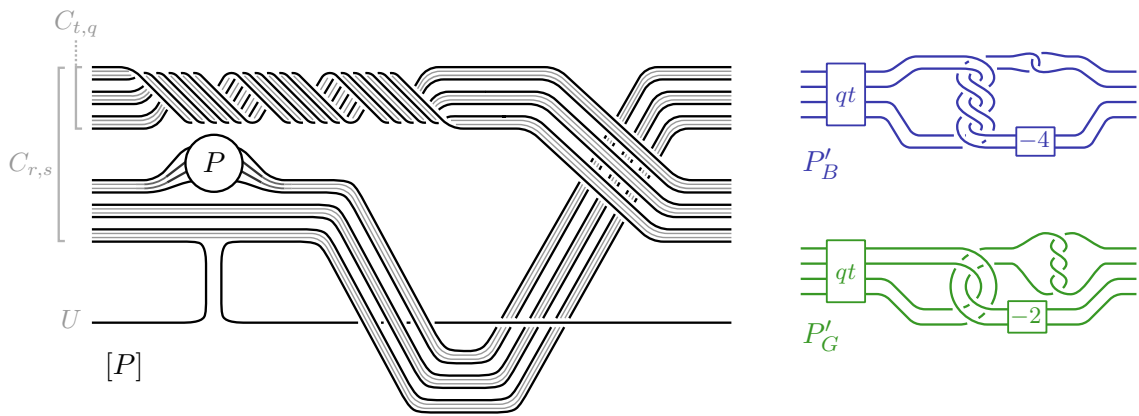
    \caption{The template knot $[P]$ is the band sum of the satellite knot $P(C_{-t,-q}(T_{r,s}))$ with an unknot $U$ situated such that $T_{r,s}=C_{r,s}(U)$.   Here, $r,s,t$ are integers such that $ps-qr=1$ and $t=-s(1-qr)$. The knots $K_B$ and $K_G$ are given by $[P'_B]$ and $[P'_G]$, respectively, where $P'_B$ and $P'_G$ are the iterated twisted Whitehead double patterns shown.
    } 
    \label{figure:K_P_simple}
\end{figure} 

When $|p|>1$, we distinguish our pairs of knots  by comparing their (zeroth) HOMFLYPT polynomials; see Section~\ref{section:distinction}. 
Since the theorem was already known for $|p|\le 1$ by \cite{Lickorish:surgery} and \cite{Brakes} (cf. \cite{McCoy:pretzel, Wakelin}), the result follows. 
(Our $K_B$ and $K_G$ are certainly also distinct when $|p|\le 1$, but the argument would be somewhat more technical.)

\subsection{Characterising and non-characterising slopes} 

Understanding the extent to which a knot can be recovered from its Dehn surgeries relates to Thurston's hyperbolic Dehn surgery theorem \cite{Thurston,Thurston:bulletin}, which implies that generic surgeries along a hyperbolic knot indeed remember the knot complement, hence the knot itself \cite{GL}. 
In modern language, a slope $p/q \in \mathbb{Q}$ is \emph{characterising} for a knot $K \subset \mathbb{S}^3$ if the oriented homeomorphism type of the 3-manifold $\mathbb{S}^3_K(p/q)$ obtained by $p/q$-surgery on $K$ uniquely determines the knot $K$. 
For example, every slope is characterising for the unknot \cite{KMOS}, both trefoils, and the figure eight knot \cite{OS}. 
It is also known that every knot has infinitely many characterising slopes \cite{Lackenby}. 
Whilst the set of \emph{non-characterising} slopes for a given knot can contain infinitely many integers \cite{BM}, it is conjectured that only finitely many non-integers arise. 
This hints at some fundamental distinction between the constructions herein and those which produce integer non-characterising slopes.

\subsection*{Acknowledgements}

We are grateful to Cameron Gordon, Marc Kegel, Siddhi Krishna, Tye Lidman, and Danny Ruberman for helpful conversations. 
We thank the University of Texas at Austin, as well as the Fifth Duke Mathematics Journal Conference, the 2025 Georgia International Topology Conference, and the ICTP Conference on Modern Developments in Low-Dimensional Topology, for facilitating this collaboration. 

KH was supported by NSF grants DMS-1803584 and DMS-2243128. 
LP is grateful for support from the following institutes and foundations over the course of this work: NSF grant DMS-1902735, the Simons-CRM scholar-in-residence program, MPIM Bonn, SwissMAP, the Clay Foundation, the Sloan Foundation, and the Simons Collaboration grant `New Structures in Low-Dimensional Topology'. 
LW acknowledges that this work was supported by the Additional Funding Programme for Mathematical Sciences, delivered by EPSRC (EP/V521917/1) and the Heilbronn Institute for Mathematical Research (HIMR).

\section{Constructing the knots} 
\label{section:construction}

\subsection{Dehn surgery} 
\label{subsection:dehn_surgery}

Let $K \subset \mathbb{S}^3$ be a knot and let $\nu(K) = \mathbb{D}^2 \times \mathbb{S}^1$ be a closed tubular neighbourhood of $K$. 
On the torus $\partial \nu(K) = \partial \mathbb{D}^2 \times \mathbb{S}^1$, we have a \emph{meridian} $\mu_K = \partial \mathbb{D}^2 \times \{\ast\}$ and \emph{Seifert longitude} $\lambda_K$ which is the unique nullhomologous longitude in the \emph{exterior} $\mathbb{S}^3_K = \mathbb{S}^3 \setminus \text{int} \nu(K)$ of $K$. 
The classes of $\mu_K$ and $\lambda_K$ generate $H_1(\partial \nu(K); \mathbb{Z}) \cong H_1(\partial \mathbb{S}^3_K; \mathbb{Z})$ and thus all simple closed curves on this torus can be expressed as a linear combination of $\mu_K$ and $\lambda_K$. 

\begin{definition}
    A \emph{framed knot} $(K,p/q)$ is a knot $K \subset \mathbb{S}^3$ together with a framing \mbox{$p/q \in \mathbb{Q} \cup \{1/0\}$}. 
    The framing corresponds to a simple closed curve $p\mu_K + q\lambda_K$ on $\partial \nu(K)$, where $(p,q)$ is a pair of coprime integers. 
\end{definition} 

Given a framed knot, we can build a new closed orientable 3-manifold in the following way. 

\begin{definition}
    The \emph{Dehn surgery on $K$ of slope $p/q$}, which we will denote  $\mathbb{S}^3_K(p/q)$, is the 3-manifold obtained by gluing a (parametrised) solid torus $V = \mathbb{D}^2 \times \mathbb{S}^1$ to the knot exterior $\mathbb{S}^3_K$ via a map $A: \partial V \to \partial \mathbb{S}^3_K$ which sends the meridian $\mu_V = \partial \mathbb{D}^2 \times \{\ast\}$ of $V$ to the curve $p\mu_K + q\lambda_K$ on $\partial \mathbb{S}^3_K$. 
\end{definition}    

Notice that the gluing map $A$ is an element of the mapping class group $Mod(\partial V)\cong SL(2,\mathbb{Z})$ of the torus, which is a more sophisticated group than the $\mathbb{Q}$ that appears in the notation. 
It is a routine exercise to check that specifying $A(\mu_V)$ is sufficient to determine the 3-manifold $\mathbb{S}^3_K(p/q)$, so one can use in fact \emph{any} $A\in SL(2,\mathbb{Z})$ encoded by a matrix
\[A = 
    \begin{pmatrix} 
        p & r \\ 
        q & s 
    \end{pmatrix}
\] 
with respect to the bases $(\mu_V, \lambda_V)$ of $\partial V$ and $(\mu_K, \lambda_K)$ of $\partial \mathbb{S}^3_K$, where $\lambda_V$ is the longitude $\{pt\}\times S^1$ given by the parametrisation of $V$. 

We emphasise that $(r,s)$ can be \emph{any} pair of coprime integers satisfying $ps-qr=1$.
For the constructions in this paper, we will (abnormally) need to keep track of the choice of $(r,s)$. 
When we need to establish a particular $(r,s)$, we will use the notation $\EK(A)$. 
For the remainder of the paper, we assume that $p/q$ is fixed and a choice of $(r,s)$ has been made.

\subsection{Duals and double duals}
\label{subsection:double_duals}

We will now discuss dual knots, which we use to ``undo'' a Dehn surgery.

\subsubsection{The dual of a knot}
\label{subsubsection:dual_knot}

Inside a Dehn surgery $\EK(A)$, the core $c$ of the surgery solid torus $V$ becomes a knot $K^\ast$, which we call the \emph{dual knot} to $K$. 

Our interest in dual knots will stem from the fact that one can perform another Dehn surgery on $K^\ast \subset \mathbb{S}^3_K(A)$ to recover $\mathbb{S}^3$. 
To see this, note that $K^\ast$ inherits a natural framing from the parametrisation of the surgery solid torus $V$. 
With respect to this framing, if we perform Dehn surgery on $K^\ast\subset\EK(A)$ with \emph{dual gluing map} $A^{-1}$, then the consecutive Dehn surgeries reduce to the do-nothing operation; hence we recover $\mathbb{S}^3$.

For the purposes of this paper, we will want to think about $K^\ast\subset \EK(A)$ isotoped out of the surgery solid torus $V$ and into $\mathbb{S}^3_K$. 
We will also need to keep track of the dual gluing map $A^{-1}$ under this isotopy and describe it in terms of the Seifert framing\footnote{We equip $K^\ast\subset\EK$ with a Seifert framing via the natural embedding $\EK\hookrightarrow \mathbb{S}^3$.} of $K^\ast$. 
This is the content of the following lemma. 

\begin{lemma}
\label{lemma:abstract_nonsense}
    The dual knot $K^\ast \subset \EK(A)$ is isotopic to the cable knot $C_{r,s}(K) \subset \EK$ and the dual gluing map is given by
        \[A^\ast = 
            \begin{pmatrix} 
                s(1-qr) & qr^2 \\ 
                -q & p
            \end{pmatrix}. 
        \] 
\end{lemma}

In the proof of Lemma \ref{lemma:abstract_nonsense}, and throughout the paper, we will often encounter framed cable knots, for which the following well-known fact will be helpful. 
Our convention is to let $C_{r,s}$ denote the cable pattern with winding number $s$. 

\begin{lemma}
\label{lemma:surface_framing}
    Let $K$ be a knot in $\mathbb{S}^3$ with tubular neighbourhood $\nu(K)$ and let $C_{r,s}(K)$ denote the $(r,s)$-cable of $K$, viewed as a curve on $\partial \nu(K)$. 
    The framing of $C_{r,s}(K)$ induced by the surface $\partial \nu(K)$ is the $rs$-framing when measured relative to the Seifert framing of $C_{r,s}(K)$.
\end{lemma}

\begin{proof}[Proof of Lemma \ref{lemma:abstract_nonsense}]
    The core $c$ of $V$ can be isotoped onto $\lambda_V\subset\partial V$.
    We know that $A$ maps $\lambda_V$ to $r\mu_K + s\lambda_K$. 
    Hence $A$ maps $c$ to the knot corresponding to this simple closed curve on $\partial \nu(K)$, which is precisely the $(r,s)$-cable $K^\ast = C_{r,s}(K)$ of $K$. 
    
    To find the dual gluing map, we write the composition of Dehn surgeries as
    \[\mathbb{S}^3 = \mathbb{S}^3_K(p/q)_{K^\ast} \cup V^\ast = \mathbb{S}^3_K \cup V_c \cup V^\ast\]
    where $V_c = V \setminus \text{int} \nu(c)$. 
    As discussed, the gluing of $V^\ast$ into $V_c$ is encoded by the matrix 
    \[A^{-1} =
        \begin{pmatrix} 
            s & -r \\ 
            -q & p 
        \end{pmatrix} 
    \] with respect to the bases $(\mu_{V^\ast}, \lambda_{V^\ast})$ of $\partial V^\ast$ and $(\mu_c, \lambda_c)$ of $\partial V_c \setminus \partial V$. 
    
    To obtain the matrix $A^{\ast}$, we must express the matrix $A^{-1}$ with respect to the basis $(\mu_{K^\ast}, \lambda_{K^\ast})$ arising from the Seifert framing of $K^{\ast}$. 
    By Lemma ~\ref{lemma:surface_framing}, the cable surface framing can be expressed as \mbox{$\lambda_{K^\ast}' = rs\mu_{K^\ast} + \lambda_{K^\ast}$}. 
    Multiplying $A^{-1}$ on the left by 
    \[Z = 
    \begin{pmatrix}
        1 & rs \\ 
        0 & 1
    \end{pmatrix}
    \]
    gives $A^{\ast} = Z A^{-1}$. 
\end{proof}

\subsubsection{The double dual of a knot}
\label{subsubsection:double_dual_knot}

Iterating this process, we can think about the image of the core $c^\ast$ of the solid torus $V^\ast$ in $\mathbb{S}^3_K(A)_{K^\ast}(A^\ast)$ as the dual of $K^\ast$ and thus the \emph{double dual knot} $K^{\ast\ast}$ to $K$. 
Similarly to above, we have that $A$ is the \emph{double dual gluing map}, namely 
\[\EK(A)_{K^\ast}(A^{-1})_{K^{\ast\ast}}(A)\cong \EK(A).\] 

As before, it will be convenient for us to view $K^{\ast\ast}$ isotoped out of the solid torus $V^\ast$ and into $\EK(A)_{K^\ast}$; see Figure~\ref{fig:pushout}. Since $\EK(A)_{K^\ast}$ is a Dehn filling of $\mathbb{S}^3_{K,K^\ast}$, by forgetting the surgery solid torus $V$ we can think of $K^{\ast\ast}\subset\mathbb{S}^3_{K,K^\ast}$. 
In this way, $(K,K^\ast,K^{\ast\ast})$ can be regarded as a 3-component link in $\mathbb{S}^3$. 
We will show by a similar argument that, after this isotopy, $K^{\ast\ast}$ is a cable of $K^\ast$, and thus an iterated cable of $K$. 
We also redescribe the double dual gluing map for $K^{\ast\ast}$ in terms of the Seifert framing in $\mathbb{S}^3_{K,K^\ast}\subset \mathbb{S}^3$.  

\begin{figure}[htbp!]
    \center
    \includegraphics[width=\linewidth]{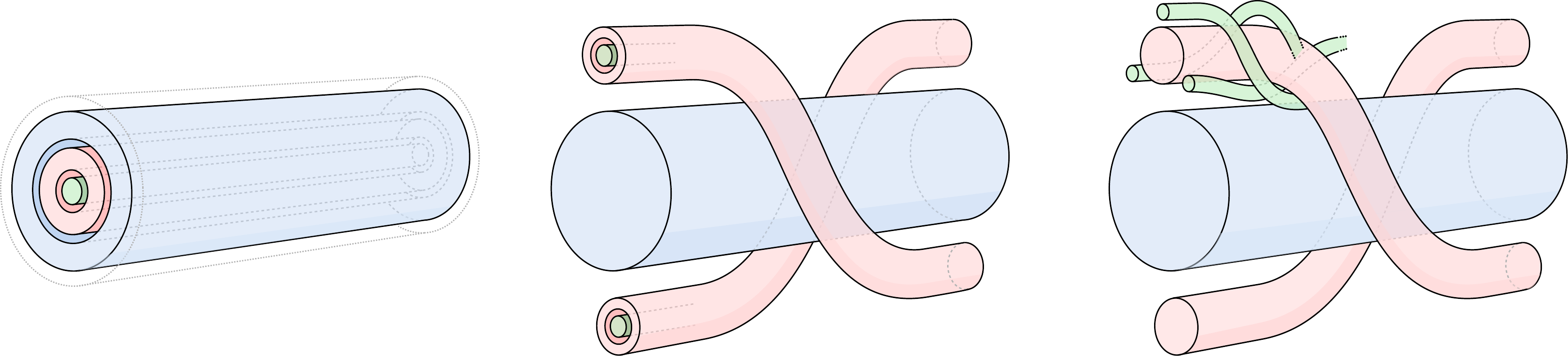}
    \caption{Perspectives on duals and double duals.}
    \label{fig:pushout}
\end{figure}

\begin{lemma}
\label{lemma:more_abstract_nonsense}
    The double dual knot $K^{\ast\ast} \subset \mathbb{S}^3$ is isotopic to the iterated cable knot $C_{qr^2,p}(C_{r,s}(K)) \subset \mathbb{S}^3_{K^\ast}$ and the double dual gluing map is given by
        \[A^{\ast\ast} = 
            \begin{pmatrix} 
                p(1+q^2r^2) & r(1+pqrs) \\ 
                q & s
            \end{pmatrix}. 
        \] 
\end{lemma}

\begin{proof}
    Performing Dehn surgery on $K^\ast$ with the matrix $A^{\ast}$ from Lemma \ref{lemma:abstract_nonsense} should map the core $c^\ast$ of $V^\ast$ to the image of $\lambda_{V^\ast}$ in $\partial \mathbb{S}^3_K(p/q)_{K^\ast}$, which is simply $qr^2\mu_{K^\ast} + p\lambda_{K^\ast}$. 
    We can thus express the double dual as the cable $K^{\ast\ast} = C_{qr^2,p}(K^\ast) = C_{qr^2,p}(C_{r,s}(K))$. 

    By the same reasoning as in the proof of Lemma \ref{lemma:abstract_nonsense}, we glue $V^{\ast\ast}$ into $V^\ast_{c^\ast} = V^\ast \setminus \text{int}\nu(c^\ast)$ via the matrix 
    \[A = 
        \begin{pmatrix} 
            p & r \\ 
            q & s 
        \end{pmatrix}
    \] 
    with respect to the bases $(\mu_{V^{\ast\ast}}, \lambda_{V^{\ast\ast}})$ of $\partial V^{\ast\ast}$ and $(\mu_{c^\ast}, \lambda_{c^\ast})$ of $\partial V^\ast_{c^\ast} \setminus \partial V^\ast$. 

    To obtain the double dual matrix $A^{\ast\ast}$, we must express $A$ with respect to the basis $(\mu_{K^{\ast\ast}}, \lambda_{K^{\ast\ast}})$ arising from the Seifert framing of $K^{\ast\ast}$. 
    By Lemma~\ref{lemma:surface_framing}, the cable surface framing can be expressed as \mbox{$\lambda'_{K^{\ast\ast}} = pqr^2\mu_{K^{\ast\ast}} + \lambda_{K^{\ast\ast}}$}. 
    Multiplying $A$ on the left by 
    \[Z' = 
    \begin{pmatrix}
        1 & pqr^2 \\ 
        0 & 1
    \end{pmatrix}\]
    gives $A^{\ast\ast} = Z'A$. 
\end{proof}

\begin{remark}
\label{remark:dualframings}
    We close this section with a remark that, whilst we have kept track of the entire gluing maps throughout this section, the results remain true when one only remembers the usual (rational) framing data. 
    In particular, we have established here that $L=(K, K^\ast, K^{\ast\ast})$, thought of as a 3-component link in $\mathbb{S}^3$, has the property that:  
    \begin{align*}
        \mathbb{S}^3_{K,K^{\ast},K^{\ast\ast}}(p/q, -s(1-qr)/q, 1/0) &\cong \mathbb{S}^3; \\ 
        \mathbb{S}^3_{K,K^{\ast},K^{\ast\ast}}(1/0, -s(1-qr)/q, p(1+q^2r^2)/q) &\cong \mathbb{S}^3. 
    \end{align*}
\end{remark}

\subsection{RBG links} 
\label{subsection:RBG_links}

An \emph{RBG link} is a 3-component framed link $(R,r) \cup (B,b) \cup (G,g)$ (coloured red, blue and green) satisfying the fundamental conditions that there are homeomorphisms 
\begin{align*}
    \phi_G&: \mathbb{S}^3_{R,B}(r,b) \longrightarrow \mathbb{S}^3 \\
    \phi_B&: \mathbb{S}^3_{R,G}(r,g) \longrightarrow \mathbb{S}^3
\end{align*}
and perhaps satisfying additional technical conditions. 

RBG links naturally produce pairs of knots which share a Dehn surgery in the following way. 
Consider the 3-manifold $M := \mathbb{S}^3_{R,B,G}(r,b,g)$ obtained from surgering all three components. 
If we imagine doing the first two surgeries first, then we obtain $\mathbb{S}^3$ by hypothesis. 
The final surgery to yield $M$ is then done on some (green) knot $K_G := \phi_G(G) \subset \mathbb{S}^3$, so that $M \cong \mathbb{S}^3_{K_G}(g')$. 
However, we can also imagine doing the first and third surgeries first. 
This gives $M \cong \mathbb{S}^3_{K_B}(b')$ for some (blue) knot $K_B := \phi_B(B) \subset \mathbb{S}^3$.
More conditions are sometimes added to the definition of RBG links to control the new framings  $b'$ and $g'$.

Usually, RBG links are used to produce pairs of knots that share an integer surgery, but the construction works for rational surgeries as well. 
In what follows, we describe an RBG construction that produces a pair of knots that share a $p/q$-surgery. 
Note that this works for \emph{any} rational $p/q$.

\subsubsection{The initial RBG link}
\label{subsubsection:initial_RBG}

Using the framed duals and double duals that we set up in Subsection \ref{subsection:double_duals}, we construct an initial RBG link $\widehat{L}$ as depicted in Figure \ref{figure:initial_RBG}:
\begin{align*}
    (\widehat{B}, \hat{b}) &= (U, p/q); \\
    (\widehat{R}, \hat{r}) &= (C_{r,s}(\widehat{B}), -s(1-qr)/q); \\
    (\widehat{G}, \hat{g}) &= (C_{qr^2,p}(\widehat{R}), p(1+q^2 r^2)/q). 
\end{align*} 

\begin{figure}[htbp!] 
    \centering 
    \def\svgwidth{.55\linewidth}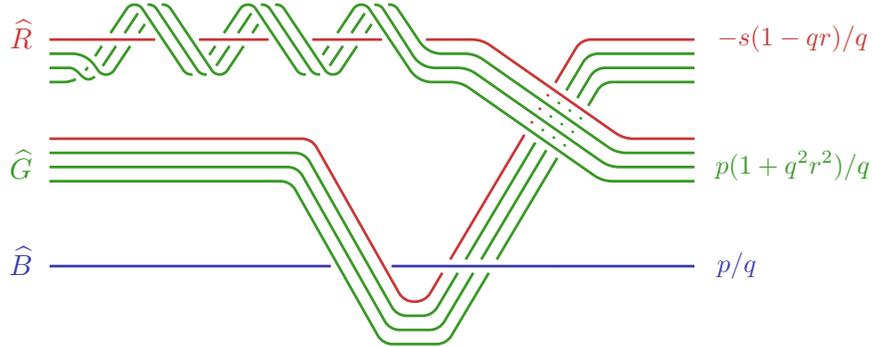
    \caption{The initial RBG link $\widehat{L} = \widehat{R} \cup \widehat{B} \cup \widehat{G}$.}
    \label{figure:initial_RBG}
\end{figure} 

As discussed in Remark \ref{remark:dualframings}, $\widehat{L}$ has the properties of an RBG link: the surgeries $(\widehat{B}, \hat{b})$ and $(\widehat{R}, \hat{r})$ cancel to leave a surgery on a single green knot $K_{\widehat{G}}$ and the surgeries $(\widehat{R}, \hat{r})$ and $(\widehat{G}, \hat{g})$ cancel to leave a surgery on a single blue knot $K_{\widehat{B}}$. 
In the following lemma, we confirm that both of these new knots have the same framing $p/q$. 

\begin{proposition}
\label{proposition:initialRBG}
    The link $\widehat{L}$ is an RBG link presenting knots $K_{\widehat{B}}$ and $K_{\widehat{G}}$ for which $\mathbb{S}^3_{K_{\widehat{B}}}(p/q) \cong \mathbb{S}^3_{K_{\widehat{G}}}(p/q)$. 
\end{proposition}

\begin{proof} 
    By Lemmas \ref{lemma:abstract_nonsense} and \ref{lemma:more_abstract_nonsense}, the link $\widehat{L}$ was constructed so that there are homeomorphisms 
    \begin{align*}
        \phi_{\widehat{G}}: \mathbb{S}^3_{\widehat{R},\widehat{B}}(\hat{r},\hat{b}) &\longrightarrow \mathbb{S}^3, \quad \phi_{\widehat{G}}((\widehat{G},\hat{g})) = (K_{\widehat{G}},\hat{g}'); \\
        \phi_{\widehat{B}}: \mathbb{S}^3_{\widehat{R},\widehat{G}}(\hat{r},\hat{g}) &\longrightarrow \mathbb{S}^3, \quad \phi_{\widehat{B}}((\widehat{B},\hat{b})) = (K_{\widehat{B}},\hat{b}'). 
    \end{align*}
    We must show that these homeomorphisms induce the framings $\hat{b}'=\hat{g}'=p/q$ on these knots. 
    
    The homeomorphism $\phi_{\widehat{B}}$ given by the dual surgeries $(\widehat{R}, \hat{r})$ and $(\widehat{G}, \hat{g})$ does not affect $(\widehat{B}, \hat{b})$ because $\widehat{G}$ lives in a neighbourhood of $\widehat{R}$ that can be isotoped away from $\widehat{B}$. 
    Hence we are left with precisely $\phi_{\widehat{B}}((\widehat{B}, \hat{b})) = (K_{\widehat{B}}, p/q)$. 
    
    The homeomorphism $\phi_{\widehat{G}}$ given by the dual surgeries $(\widehat{B}, \hat{b})$ and $(\widehat{R}, \hat{r})$ does however affect $(\widehat{G}, \hat{g})$ because $\widehat{R}$ lives in a neighbourhood of $\widehat{B}$ which can never be isotoped away from $\widehat{G}$. 
    To see that the framing on $K_{\widehat{G}}$ is also $p/q$, it is useful to pass back to the perspective where $\widehat{R} = K_{\widehat{B}}^\ast \subset V$ and $\widehat{G} = \widehat{K}_{B}^{\ast\ast} \subset V^\ast$. 
    Then we see that the surgery solid torus for $K_{\widehat{G}}$ is glued into $\mathbb{S}^3$ with gluing map $A$, and $A(\mu_{V^{\ast\ast}})=p/q$, so $\phi_{\widehat{G}}((\widehat{G}, \hat{g})) = (K_{\widehat{G}}, p/q)$. 
\end{proof}

Unfortunately, the simplicity of the initial RBG link $\widehat{L}$ implies that $K_{\widehat{B}}=U=K_{\widehat{G}}$. 
These knots do of course have common $p/q$-surgeries, which is comforting, but to prove Theorem \ref{theorem:non-characterising} we will require some modifications to obtain knots that are not isotopic.

\subsubsection{The modified RBG link}
\label{subsubsection:final_RBG}

To obtain a less symmetric RBG link, we modify $\widehat{L}$ by taking a band sum of $\widehat{B}$ with a particular choice of satellite knot $P_B(\mu_G)$, as depicted on the left-hand side of Figure \ref{figure:final_RBG}. 
This gives the new link $L = R \cup B \cup G$: 
\begin{align*}
    (B,b) &= (\widehat{B} \#_{\beta} P_B(\mu_G), p/q); \\ 
    (R,r) &= (\widehat{R}, -s(1-qr)/q); \\ 
    (G,g) &= (\widehat{G}, p(1+q^2 r^2)/q). 
\end{align*} 

\begin{figure}[htbp!]
    \centering
    \def\svgwidth{.975\linewidth}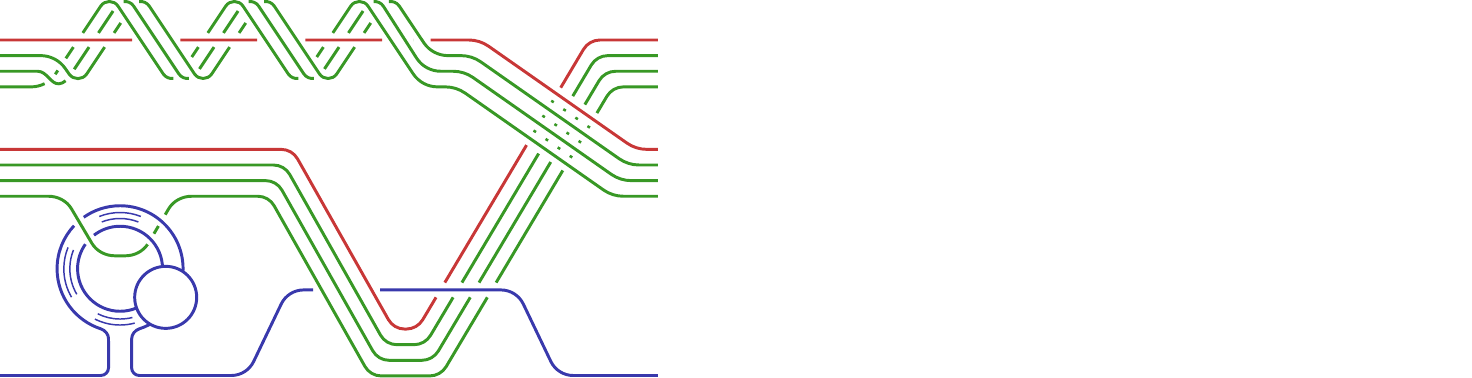
    \caption{The modified RBG link, drawn in two different ways.}
    \label{figure:final_RBG}
\end{figure}

To define the satellite pattern, let $D^k_m$ denote the $k$-clasped, $m$-twisted Whitehead double pattern as illustrated in Figure~\ref{figure:whitehead}. 
(For $m=0$, we may sometimes omit the subscript, simply writing $D^k$.) 
We set $P_B = D^1_0 \circ D^2_0$, an iterated Whitehead pattern. 
Note that this is an unknotted asymmetric pattern with winding number zero and wrapping number four. 
We can naturally consider $P_B$ as a pattern in a solid torus $\mathbb{S}^3_U$. 
From $P_B$ we can obtain another pattern $P_G$ by exchanging the components of $U \cup P_B$ by a component-preserving isotopy, yielding $P_G \cup U$. 
The induced pattern $P_G = D^2_0 \circ D^1_0$ also has winding number zero and wrapping number four. 

\begin{figure}[htbp!]
    \centering 
    \def\svgwidth{.3\linewidth}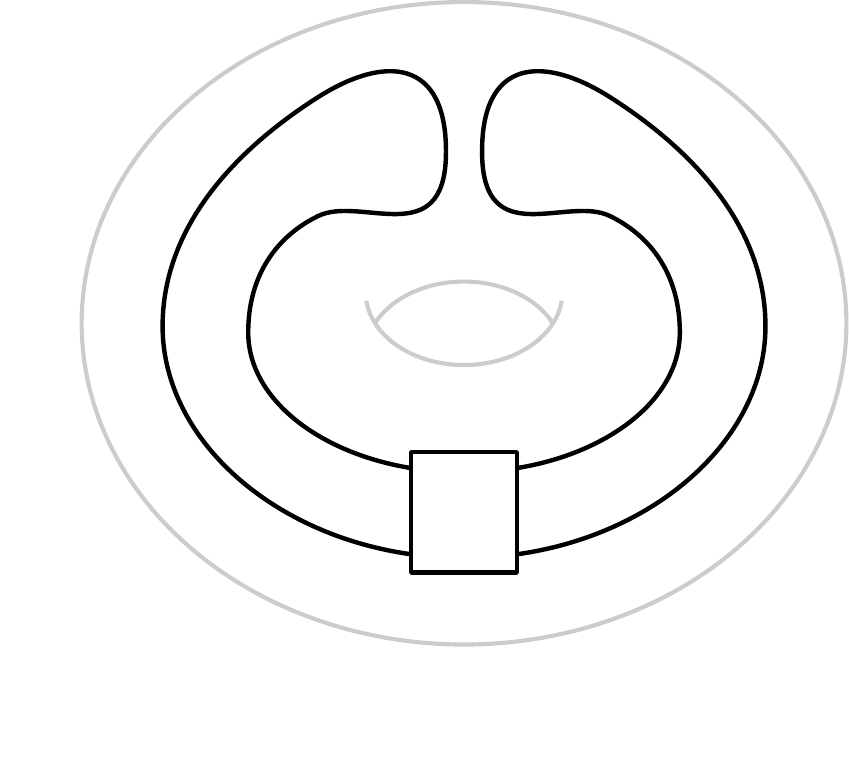
    \caption{The $k$-clasped $m$-twisted Whitehead doubling pattern $D^k_m$ in a (parametrised) solid torus, viewed as the exterior of the unknot $U$. Here, $k, m \in \mathbb{Z}$ denote numbers of full right-handed twists between parallel strands.}
    \label{figure:whitehead}
\end{figure} 

\begin{figure}[htbp!]
    \centering 
    \def\svgwidth{.725\linewidth}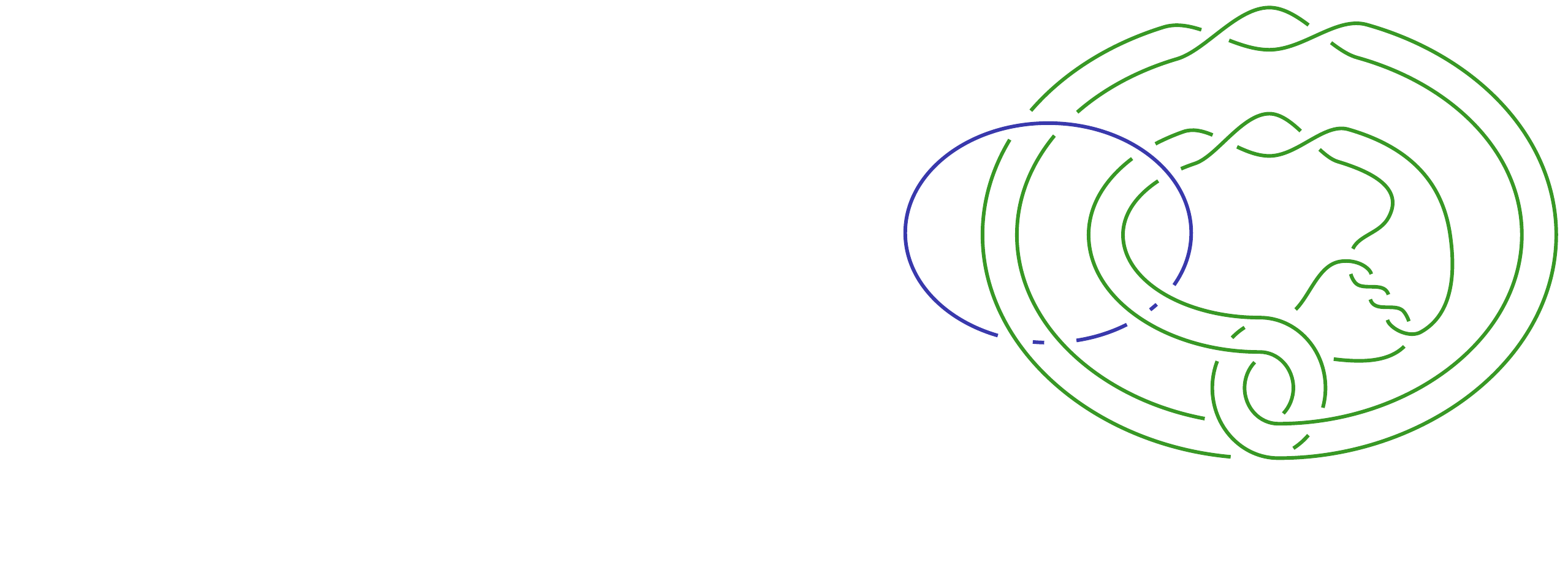
    \caption{The patterns $P_B=D^1_0 \circ D^2_0$ and $P_G=D^2_0 \circ D^1_0$, each lying in a solid torus expressed as the exterior of the unknot $U$.}
    \label{figure:patterns}
\end{figure} 

At the level of $L$, the isotopy from $P_B$ to $U$ straightens out $B$ at the expense of $G$ (without affecting $R$), as on the right-hand side of Figure \ref{figure:final_RBG}. 
Using this, we see that $L$ is isotopic to: 
\begin{align*}
    (B,b) &= (\widehat{B}, p/q); \\ 
    (R,r) &= (\widehat{R}, -s(1-qr)/q); \\
    (G,g) &= (\widehat{G} \#_{\beta} P_G(\mu_B), p(1+q^2 r^2)/q). 
\end{align*} 

We show now that $L$ is also an RBG link with $b'=g'=p/q$. 

\begin{proposition}
\label{prop:framing_still_good}
    The link $L$ is an RBG link presenting knots $K_B$ and $K_G$ for which $\ts_{K_B}(p/q)\cong\ts_{K_G}(p/q)$.
\end{proposition}

\begin{proof}
    To obtain $L$, we modified the geometric linking (but not the algebraic linking!) of $\widehat{B}$ with $\widehat{G}$. 
    Since $\widehat{L}$ was an RBG link and we have not disturbed the red and blue (nor red and green) sublinking types or framings, $L$ still satisfies the fundamental conditions to be an RBG link. 
    Note then that $\smash{\phi_B=\phi_{\widehat{B}}}$ and $\smash{\phi_G=\phi_{\widehat{G}}}$.
    Furthermore, since we did not change the algebraic linking of $\widehat{B}$ with $\widehat{G}$, the common surgery slope for $K_B$ and $K_G$ is as it was in Proposition \ref{proposition:initialRBG}, which was $p/q$. 
\end{proof} 

Since we have broken the symmetry of $\widehat{L}$, we can be optimistic that this new link $L$ presents distinct knots $K_B$ and $K_G$. 
To this end, we now give concrete descriptions of $K_B$ and $K_G$ (see also Figure~\ref{figure:K_B-and-K_G}). 
To do so, we define a template as follows. 
Let $\overline{C}_{t,q} := C_{-t,-q}$ denote the cable pattern $C_{t,q}$ with its strand orientation reversed. 
For any satellite pattern $P$, define 
\[[P] := U \, \#_{\beta} \, P(\overline{C}_{t,q}(T_{r,s}))\] 
to be the band sum of the unknot $\smash{\widehat{B}}=U$ with the satellite $P(\overline{C}_{t,q}(T_{r,s}))$, where $t=-s(1-qr)$, the unknot $\widehat{B}$ and the torus knot $T_{r,s}$ are linked so that $T_{r,s}=C_{r,s}(\widehat{B})$, and the particular band $\beta$ is demonstrated in Figure~\ref{figure:final_RBG}. 
We will show that $K_B=[P'_B]$ and $K_G=[P'_G]$, where $P'_B$ and $P'_G$ are the twisted patterns $P'_B=D^1_0 \circ D^2_{qt}$ and $P'_G = D^2_0 \circ D^1_{qt}$ as depicted in Figure~\ref{figure:twisted-patterns}. 

\begin{figure}[t]
    \centering
    \def\svgwidth{\linewidth}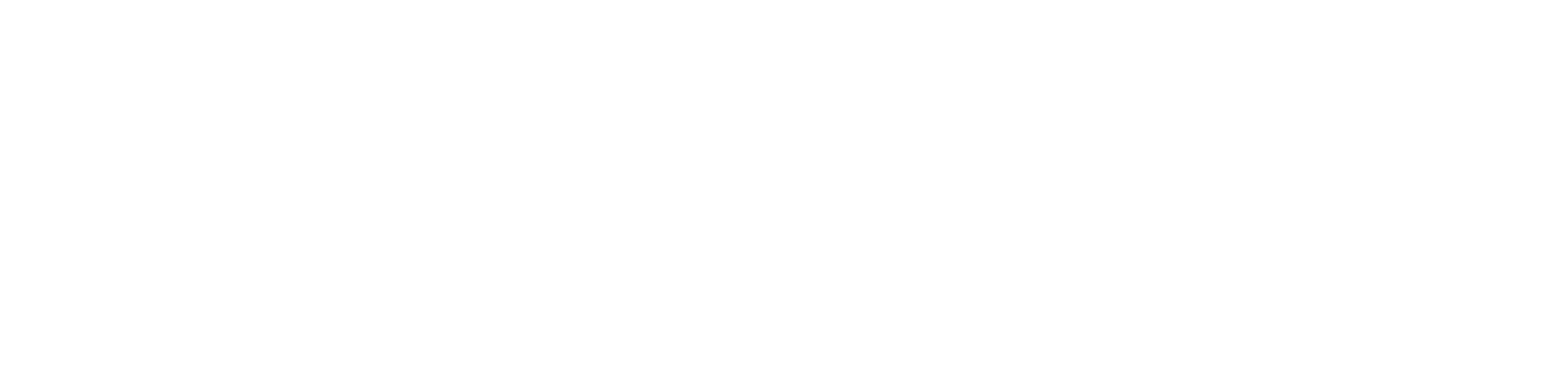
    \caption{The knots $K_B$ and $K_G$.}
    \label{figure:K_B-and-K_G}
\end{figure}

\begin{figure}[htbp!]
    \centering 
    \def\svgwidth{.725\linewidth}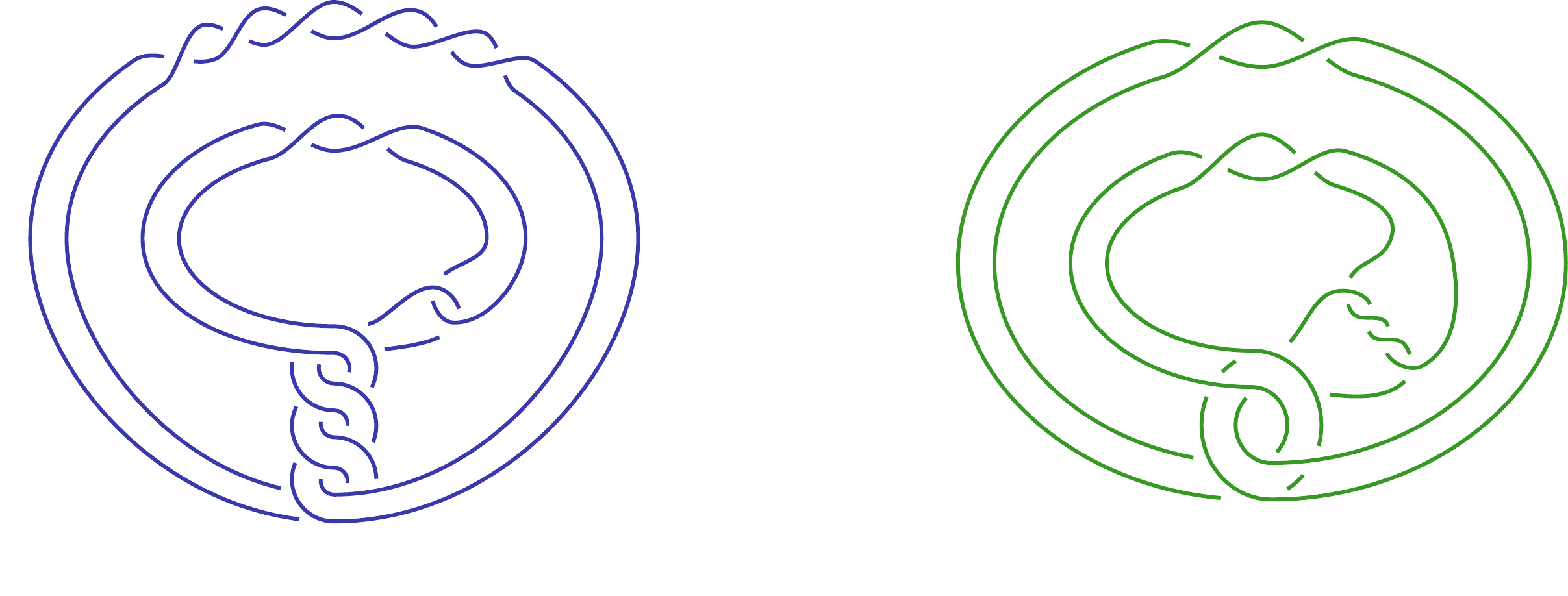
    \caption{The twisted patterns $P'_B=D^1_0 \circ D^2_{qt}$ and $P'_G=D^2_0 \circ D^1_{qt}$.}
    \label{figure:twisted-patterns}
\end{figure} 

Let us briefly digress to discuss these twisted patterns. 
Given a pattern knot $P$ in the solid torus $\mathbb{D}^2 \times \mathbb{S}^1$, the $\emph{$n$-twist}$ pattern $P_n$ is the image of $P$ after applying $n$ meridional Dehn twists to $\mathbb{D}^2 \times \mathbb{S}^1$. 

\begin{lemma}
    \label{lemma:twisting}
    The $n$-twist of the pattern $D^l_0 \circ D^k_0$ is the pattern $D^l_0 \circ D^k_n$.
\end{lemma}

\begin{proof}
    Recall that $D^l_0 \circ D^k_0$ can be constructed by starting with the link $L \subset \mathbb{D}^2 \times \mathbb{S}^1$ consisting of $D^k_0$ and its 0-framed pushoff\footnote{Recall that the parametrisation of the solid torus induces a well-defined $0$-framing on the pattern and a notion of writhe as follows. We require that diagrams are obtained via projection to the annulus \mbox{$\mathbb{D}^1\times\{0\}\times\mathbb{S}^1\subset \mathbb{D}^2\times \mathbb{S}^1$}, then define writhe as usual. The $0$-framing is defined to be the blackboard framing minus the writhe.} and then banding the two components together using a band with $-l$ full right-handed twists.
    Note that, since our diagram of $D^k_0$ has writhe $2k$, its 0-framed pushoff is obtained from the blackboard pushoff by introducing $-2k$ twists.
    
\begin{figure}
    \centering 
    \def\svgwidth{\linewidth}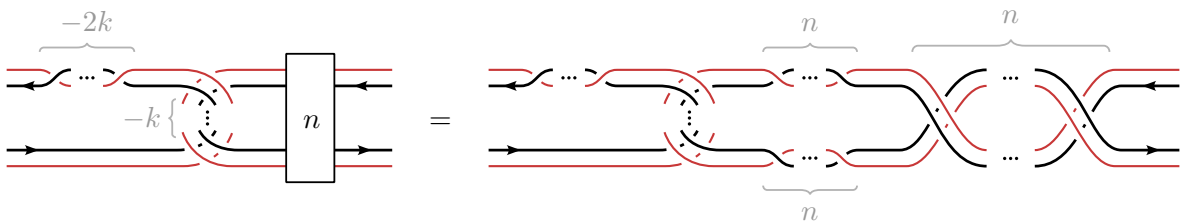
        \captionsetup{width=.9\linewidth}
    \caption{The result of applying an $n$-fold twist to the union of the pattern $D^k$ (black) and its 0-framed pushoff (red). Here, $k$ and $n$ refer to numbers of full twists.} 
    \label{figure:n-twist}
\end{figure} 

    The $n$-twist of $D^l_0 \circ D^k_0$ can be obtained from the $n$-twist of $L$ by introducing a clasp with $-l$ twists. 
    The  $n$-twist of $L$ is depicted on the left-hand side of Figure~\ref{figure:n-twist}, and the right-hand side splits the twist box into three separate twisting regions. 
    Observe that the black curve coincides with $D^k_n$. 
    Its  writhe is  $2k-2n$, and the indicated framing curve is obtained from the blackboard framing  by $-2k+2n$ twists. 
    It follows that the $n$-twist of $L$ is the union of $D^k_n$ and its 0-framed pushoff. 
    Introducing a clasp with $-l$ twists yields $D^l_0 \circ D^k_n$, as desired.
\end{proof}

\begin{proposition}
\label{proposition:finalRBG} 
    The knots $[P'_B]$ and $[P'_G]$ described above (and in Figure \ref{figure:K_B-and-K_G}) agree with the knots $K_B$ and $K_G$ presented by the $RBG$ link $L$ from Proposition \ref{prop:framing_still_good}.
\end{proposition} 

Our proof of Proposition \ref{proposition:finalRBG} will require pushing curves in and out of the red, blue, and green surgery solid tori. 
To simplify the diagrammatics, we will assume the gluing maps $A$ and $A^{-1}$ are expressed as products of Dehn twists along a fixed choice of meridian and longitude. 
This ensures that the gluing maps are supported in a neighbourhood of the chosen meridian--longitude pair, represented by the grey regions in Figure~\ref{figure:twisting}. 
In particular, this figure illustrates the effect of these gluing maps on another meridian lying outside the given neighbourhood. 
In our diagrams, we will not be careful with the direction of meridional twisting, but the direction of longitudinal twisting will be essential. 
Parts (b) and (c) of Figure~\ref{figure:twisting} represent the image of the meridian under a gluing map that sends the meridian to a curve with respectively positive and negative longitudinal component. 
Here, we assume that all strands in the grey region are coherently oriented. 

\begin{figure}[htbp!]
\smallskip
    \centering 
    \def\svgwidth{.9\linewidth}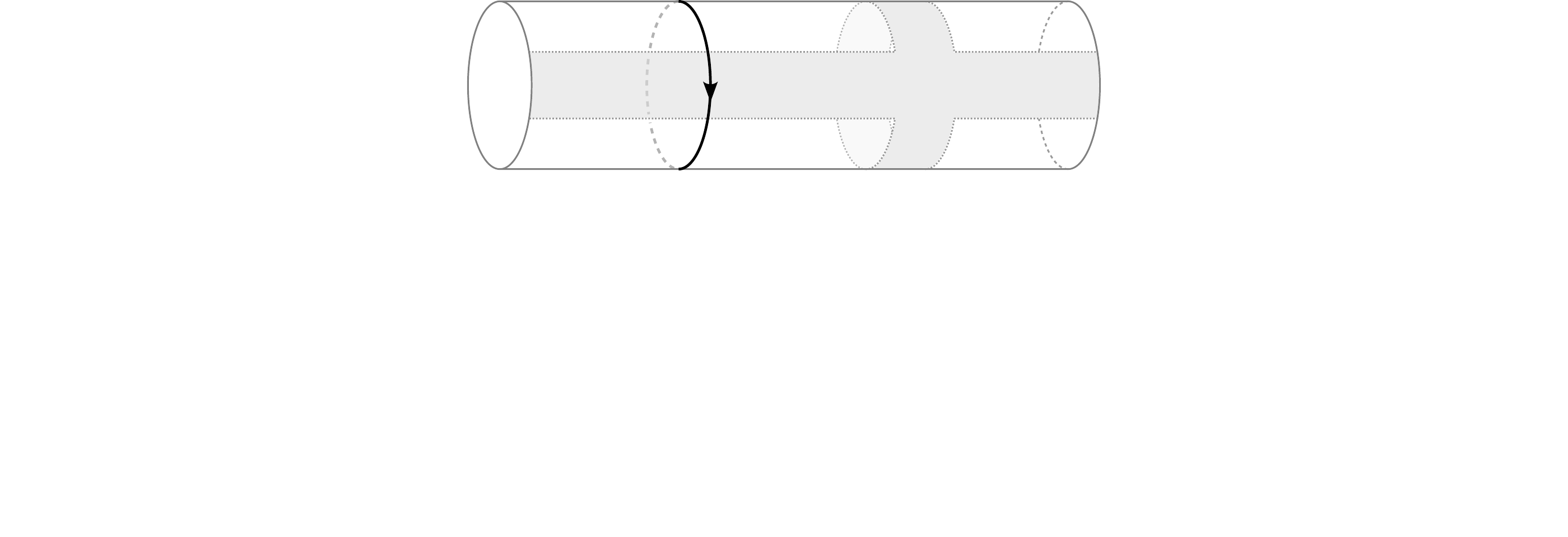
    \caption{Gluing maps are supported in the region indicated.} 
    \label{figure:twisting}
\end{figure} 

\begin{proof}[Proof of Proposition~\ref{proposition:finalRBG}]

\begin{figure}[htbp!]
    \centering 
    \def\svgwidth{\linewidth}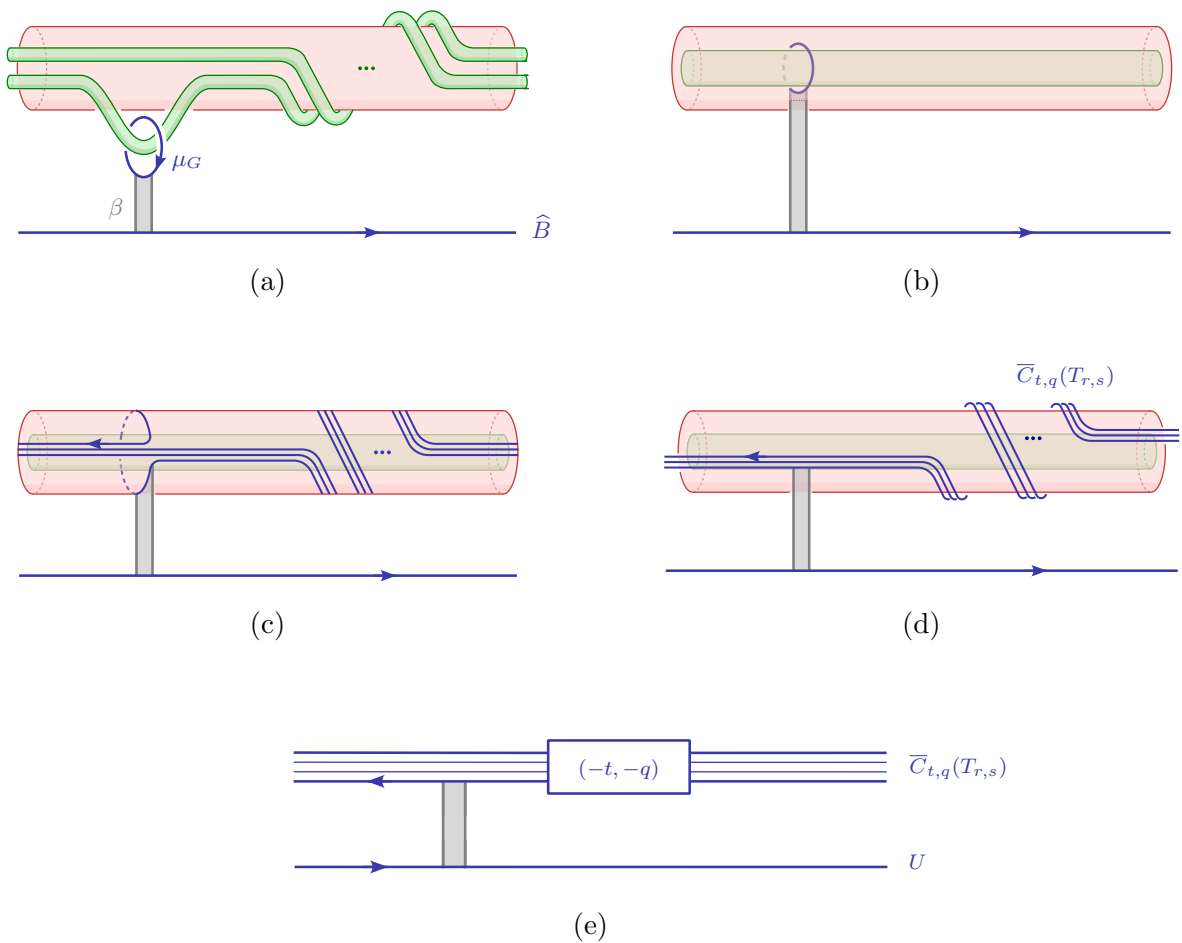
    \caption{(a) The band sum of $\widehat B$ and $\mu_G$. (b) After pushing the green solid torus into the blue solid torus $V^*$, the curve $\mu_G$ becomes a meridian of the core of $V^*$. (c) Pushing $\mu_G$ back out of the red solid torus to a curve lying on $\partial V^*$. (d) This is a $(-t,-q)$-cable, i.e.~$\overline{C}_{t,q}(T_{r,s})$. (e) The band sum of $\widehat B=U$ and the image of $\mu_G$ along the band $\beta$ after cancelling the surgeries along the red and green tori. } 
    \label{figure:dance-3}
\end{figure} 

    Consider the first description of $L$, where $B = \widehat{B} \#_{\beta} P_B(\mu_G)$. 
    Here $B$ is described essentially in three parts: $\widehat{B},P_B(\mu_G),$ and the band $\beta$. 
    To describe $K_B$, we will push all three of these pieces through the homeomorphism $\phi_B$, and reassemble their images. 
    
    The homeomorphism $\phi_B: \mathbb{S}^3_{R,G}(r,g) \to \mathbb{S}^3$ does not disturb $\widehat{B}$. 
    We claim that $\phi_B$ sends $\mu_G$ to the cable $\smash{\overline{C}_{t,q}(R)}$. 
    To see this, isotope $G$ back into the surgery solid torus $V^\ast$ for $R$ (the inverse of Proposition \ref{lemma:more_abstract_nonsense}), where it becomes $c^\ast$, and $\mu_G$ becomes $\mu_{c^\ast}$; see parts (a) and (b) of Figure~\ref{figure:dance-3}. 
    We can then isotope $\mu_{c^\ast}$ back out of $V^\ast$ as in parts (c) and (d) of Figure~\ref{figure:dance-3}; the image of $\mu_{c^\ast}$ is 
    \[A^\ast\begin{pmatrix} 1 \\ 0 \end{pmatrix}=\begin{pmatrix} -t \\ -q \end{pmatrix}.\]
    As a curve on the red torus $\partial V^\ast$, this is precisely $\overline{C}_{t,q}(R)=\overline{C}_{t,q}(T_{r,s})$. 
    
    We must upgrade this to determine the image of the satellite $P_B(\mu_G)$. 
    This satellite lies in a solid torus neighbourhood of the image of $\mu_G$, so it just remains to determine how such a neighbourhood is twisted during the isotopy above. 
    To that end, we track the $0$-framing of $\mu_G$, which initially agrees with the framing of $\mu_G$ induced by the boundary of the green solid torus $V^{\ast\ast}$ in Figure~\ref{figure:dance-3}(a). 
    Next we push $V^{\ast\ast}$ into $V^{\ast}$, where it is naturally identified with a neighbourhood $\nu(c^{\ast})$ of the core curve $c^{\ast}$; the framing on $\mu_G$ now coincides with the surface framing from $\partial \nu (c^{\ast})$. 
    Thus, after dilating radially to push $\mu_G$ out of $V^{\ast}$, its framing continues to coincide with the surface framing --- now induced by $\partial V^{\ast}$. 
    Noting that the image of $\mu_G$ now coincides with a $(-t,-q)$-curve on $\partial V^{\ast}$, the surface framing corresponds to the $qt$-framing by Lemma~\ref{lemma:surface_framing}. 
    It follows that applying the untwisted pattern $P_B$ to $\mu_G$ corresponds to applying the $qt$-twisted pattern $P'_B$ from Figure~\ref{figure:twisted-patterns} to the image of $\mu_G$. 
    As argued in Lemma~\ref{lemma:twisting}, the $qt$-twist of $P_B=D^1_0 \circ D^2_{0}$ is precisely the pattern $P'_B=D^1_0\circ D^2_{qt}$. 
    Hence we have 
    \[\phi_B(P_B(\mu_G))=P'_B(\overline{C}_{t,q}(T_{r,s})).\] 

    It remains to inspect the image of the band $\beta$. 
    As discussed above, we have assumed that our gluing homeomorphisms are supported in a small neighbourhood of a fixed meridian-longitude pair. 
    We can arrange our isotopies so that the band $\beta$ passes through the complementary region where the gluing homeomorphisms coincide with the identity. 
    This ensures that the band is not disturbed when pushing or pulling parts of it through the boundaries of the surgery solid tori. 
    It then follows that the band remains diagrammatically planar and essentially  stationary as illustrated in Figure~\ref{figure:dance-3}. 
    We conclude that 
    \[
    K_B=U\#_{\beta} \, P'_B(\overline{C}_{t,q}(T_{r,s}))=[P'_B]
    \]
    as desired. 
    
    For $K_G$, we consider the second description of $L$, where $G = \widehat{G} \#_{\beta} P_G(\mu_B)$, and argue as we did for $K_B$. 
    Beginning with $\smash{\widehat{G}}$, note that in $\mathbb{S}^3_{R,B}(r,b)$, $\smash{\widehat{G}}$ is isotopic to $c^\ast$, the core of $R$, by setup. 
    After cancelling the surgeries along $B$ and $R$, this core curve $c^\ast$ is carried to where $\smash{\widehat{B}}$ previously sat. 
    Hence the homeomorphism $\phi_G: \mathbb{S}^3_{R,B}(r,b) \to \mathbb{S}^3$ takes $\smash{\widehat{G}}$ to $\smash{\widehat{B}}$. 
    
    To find $\phi_G(\mu_B)$, we begin by isotoping $\mu_B$ into the blue surgery solid torus $V$, where it becomes the 
    \[A^{-1}
    \begin{pmatrix} 
        1 \\ 0 
    \end{pmatrix}
    =
    \begin{pmatrix} 
        s \\ -q 
    \end{pmatrix}
    \] cable of some copy $c'$ of the core $c$ of $V$; see parts (a) and (b) of Figure~\ref{figure:dance-1}. 
    Next, isotope $R$ into the blue surgery solid torus $V$ as well, where it returns to being a copy of the core $c$ that is unlinked from $c'$; see parts (c) and (d) of Figure~\ref{figure:dance-1}. 
    We exchange the positions of $c$ and $c'$, then push $c'$ out of $V$ to become the 
    \[A 
    \begin{pmatrix}
        0 \\ 1 
    \end{pmatrix}
    = 
    \begin{pmatrix}
        r \\ s
    \end{pmatrix}
    \]
    cable of $\widehat{B}$. 
    Then $\mu_B$ follows to become the 
    \[
    \begin{pmatrix}
        1 & rs \\ 
        0 & 1
    \end{pmatrix}
    \begin{pmatrix}
        s \\ -q 
    \end{pmatrix}
    = 
    \begin{pmatrix}
        -t \\ -q
    \end{pmatrix}
    \]
    cable of this: namely, $\overline{C}_{t,q}(R)=\overline{C}_{t,q}(T_{r,s})$. 
    See Figure~\ref{figure:dance-1}. 

\begin{figure}[htbp!]
    \centering 
    \def\svgwidth{\linewidth}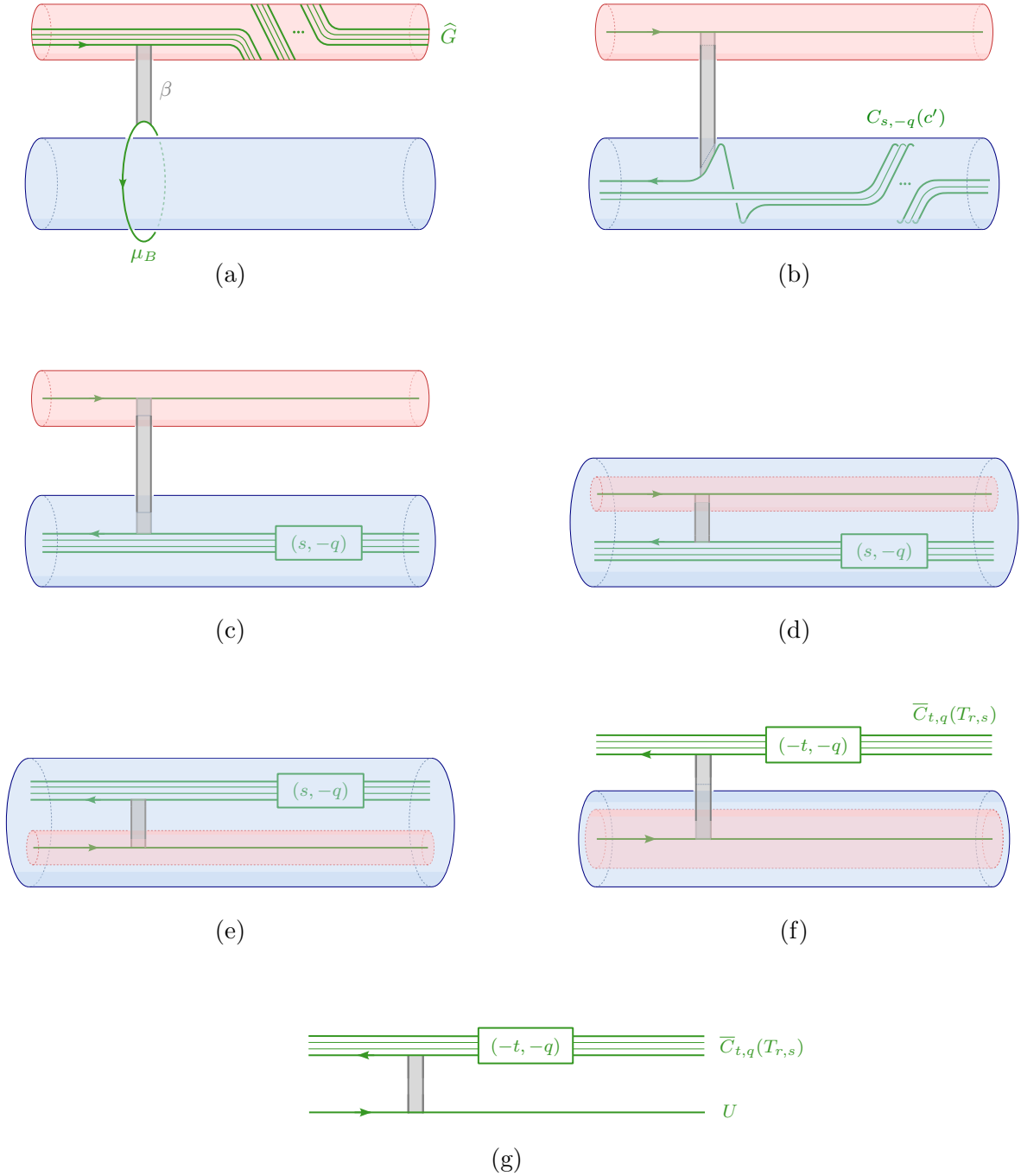
    \caption{(a) The band sum of $\widehat G$ and $\mu_B$. (b) After pushing $\mu_B$ into the blue solid torus $V$, it becomes a $(s,-q)$-cable of a copy $c'$ of the core curve $c \subset V$. Pushing $\widehat G$ into the red solid torus yields its core. (c) Redrawing $C_{s,-q}(c')$. (d) The result of pushing the red solid torus into $V$. (e) Isotopy. (f)  After pushing the $(-s,q)$-cable of $c'$ out of $V$, it becomes the $(-t,-q)$-cable of the torus knot $T_{r,s}$. (g) The gluing of the red and blue tori cancels.} 
    \label{figure:dance-1}
\end{figure} 

    \pagebreak
    
    As with our earlier analysis of $\mu_G$, we next inspect the image of the 0-framing of $\mu_B$ under the isotopies above. 
    The 0-framing of $\mu_B$ again corresponds to the framing induced by the surface $\partial V$. 
    After pushing $\mu_B$ into $V$, it becomes a $(s,-q)$-curve on a torus parallel to $\partial V$, and it remains surface-framed. 
    This torus containing the image of $\mu_B$ is shrunk to the boundary of a tubular neighbourhood $\nu(c)$ and then pushed out of $V$ longitudinally, carrying it to a tubular neighbourhood of $T_{r,s}$. 
    The image of $\mu_B$ is the $(-t,-q)$-curve on the torus $\partial \nu (T_{r,s})$ and it continues to be surface-framed. 
    The surface framing is $qt$ by Lemma~\ref{lemma:surface_framing}. 
    Therefore we conclude that the image of $P_G(\mu_B)$ is the $qt$-twisted satellite  $$\phi_G(P_G(\mu_B))=P'_G(\overline{C}_{t,q}(T_{r,s})).$$

    We can argue similarly as we did for $K_B$ that there are no substantive changes to the band under these isotopies. 
    Hence the knot $K_G$ presented by $L$ is given by \[K_G=U \#_{\beta} \, P_G(\overline{C}_{t,q}(T_{r,s}))=[P'_G]\] as desired. 
\end{proof}

The remainder of the paper is concerned with proving that $K_B \neq K_G$. 
It is unfortunate that several more pages are required to demonstrate this, as $K_B$ and $K_G$ are rather obviously distinct. 
We will do our due diligence in writing these pages, but we encourage the reader to spend the time on something more pleasant --- for example, stepping outside or petting the cat.

\section{Distinguishing the knots} 
\label{section:distinction}

To distinguish $K_B$ and $K_G$, we will need a knot invariant that is sensitive to the fact that they share a $p/q$-surgery. 
We will use an invariant derived from the HOMFLYPT polynomial which inherits a simplified skein relation. 
Since each pair of knots depends on the rational number $p/q$, it can be difficult to actually compute absolute invariants for the entire family. 
However, we will see that $K_B$ and $K_G$ admit very similar decompositions via skein relations; as such, it becomes rather simple to compute the \emph{difference} between their invariants and show that this is always non-zero.

\subsection{The zeroth coefficient polynomial} 
\label{subsection:0th} 

As a reference for this subsection, we point the reader to \cite{Ito}. 
We will recall what we need here. 

\begin{figure}[htbp!] 
    \centering 
    \medskip
    \def\svgwidth{.425\linewidth}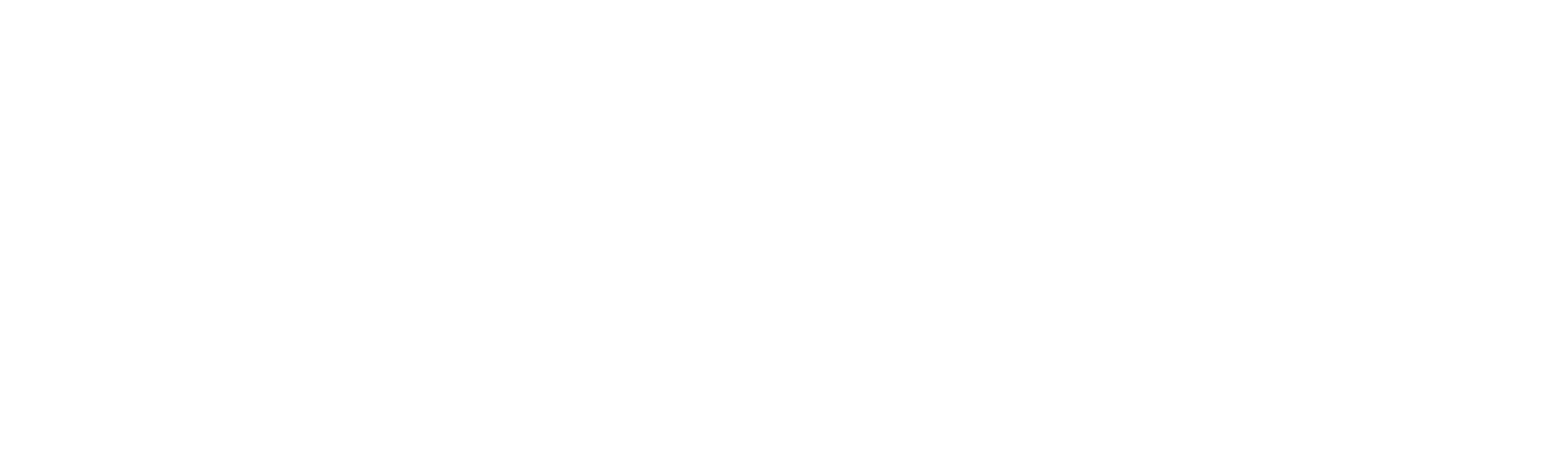
    \caption{A skein triple. } 
    \label{figure:skein}
\end{figure} 

Let $L$ be an oriented link in $\mathbb{S}^3$. 
Recall that the \emph{HOMFLYPT polynomial} $P_L(z,v) \in \mathbb{Z}[z^{\pm1}, v^{\pm1}]$ of $L$ is defined by a skein relation for skein triples $(L_+,L_0,L_-)$ as in Figure \ref{figure:skein}: 
\[v^{-1} P_{L_+}(v,z) - v P_{L_-}(v,z) = z P_{L_0}(v,z), \quad P_U(v,z)=1.\] 
There is a decomposition of this polynomial as follows: 
\[P_L(v,z) = \bigg(\frac{z}{v}\bigg)^{1-|L|} \sum_{i\geq0} p^i_L(v) z^{2i}\]
where $|L|$ is the number of components of $L$. 
Each $\Gamma^i_L(\alpha) := p^i_L(v)|_{\alpha=-v^2} \in \mathbb{Z}[\alpha^{\pm1}]$ is called the \emph{$i^\text{th}$ coefficient polynomial} of $L$. 
We will be interested in the \emph{zeroth coefficient polynomial} $\Gamma_L(\alpha) := \Gamma^0_L(\alpha)$ of $L$, which satisfies the skein relation 
\begin{equation*}
    \Gamma_{L_+}(\alpha) + \alpha\Gamma_{L_-}(\alpha) = 
    \begin{cases}
        -\alpha\Gamma_{L_0}(\alpha) & |L_{\pm}|=|L_0|-1, \\ 
        0 & |L_{\pm}|=|L_0|+1. 
    \end{cases}
\end{equation*}

The zeroth coefficient polynomial of an oriented link $L$ is determined by that of each component and the total linking number $lk(L) = \sum_{i<j} lk(K_i,K_j)$ between them. 

\begin{proposition}
\label{proposition:HOMFLYPT_link}
    Let $L = K_1 \cup \ldots \cup K_{|L|}$ be an $|L|$-component link. 
    Then 
    \begin{equation}
    \label{equation:linking}
        \Gamma_L(\alpha) = (-1)^{|L|-1}(1+\alpha^{-1})^{|L|-1} (-\alpha)^{lk(L)} \Gamma_{K_1}(\alpha) \cdots \Gamma_{K_{|L|}}(\alpha). 
    \end{equation}
\end{proposition} 
Finally, when $K$ is a knot, its zeroth coefficient polynomial satisfies $\Gamma_K(1)=1$.

\subsection{Setting up the calculation} 
We will compute the zeroth coefficient polynomials for $K_B$ and $K_G$ using the skein and linking relations discussed above. 
To set up our argument and notation, consider the rough schematic diagram of the underlying template knot \[[P]=U\#_\beta\, P(\overline{C}_{t,q}(T_{r,s}))\] in Figure~\ref{figure:skein_schematic}. 
For simplicity, we will write $C(R)=C_{t,q}(T_{r,s})$ and $\overline{C}(R)=\overline{C}_{t,q}(T_{r,s})$. 
Note that this iterated torus knot is a cable of $R$ formed by the image of each of $\mu_B$ and $\mu_G$ under the homeomorphisms $\phi_G$ and $\phi_B$ from Section~\ref{section:construction}. 
In addition, it will be convenient to write $\overline{C}_{2m,2}$ for the 2-component pattern knot obtained by reversing the orientation on \emph{one} component of the $m$-twisted 2-copy pattern $C_{2m,2}$. 

\begin{figure}[htbp!] 
    \centering 
    \def\svgwidth{.575\linewidth}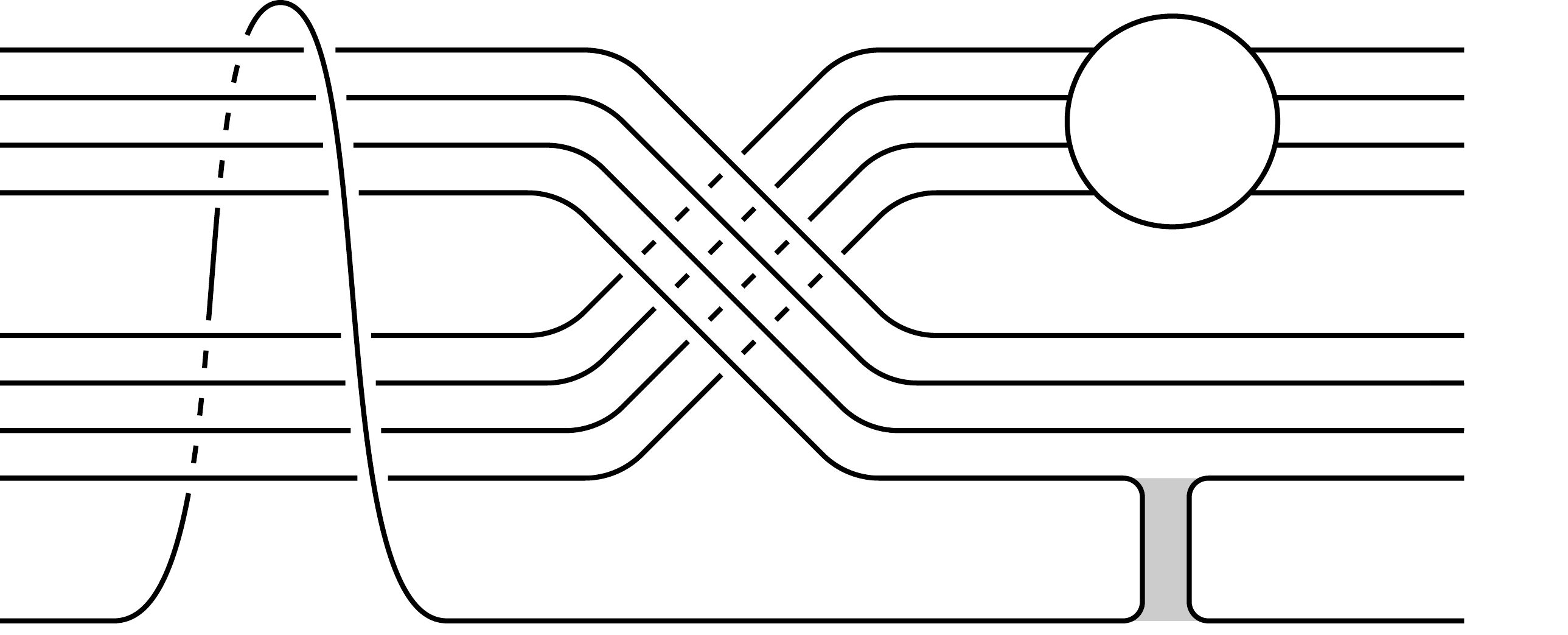
    \caption{Another diagram of the template knot $[P]=U \#_{\beta} P(\overline{C}(R))$.} 
    \label{figure:skein_schematic}
\end{figure} 

Rather than compute the entire polynomials for $K_B=[D^1 \circ D^2_{qt}]$ and $K_G=[D^2 \circ D^1_{qt}]$, we will express them in terms of the polynomials for three simpler knots: the unknot $U$, the iterated torus knot $C(R)$, and the knot $[H]=U \#_\beta \, \overline{C}(R)$ obtained by substituting the identity pattern $H=C_{0,1}$ in for $P$ in the template knot. 
We note that, since $C(R)$ is isotopic to its reverse, we have $\Gamma_{C(R)} = \Gamma_{\overline{C}(R)}$ and hence we can use these polynomials interchangeably.
For the upcoming skein/linking trees, we will repeatedly use the observation that there is a skein triple relating the Whitehead doubling patterns $L_+=D^k_m$ and $L_-=D^{k-1}_m$ and $L_0=\overline{C}_{2m,2}$, the $m$-twisted $2$-copy pattern, as in Figure~\ref{figure:skein_double}. 
(In particular, for $k=1$, we obtain a skein triple in which $L_-$ is the unknot.) 

\begin{figure}[htbp!] 
    \centering 
    \def\svgwidth{.675\linewidth}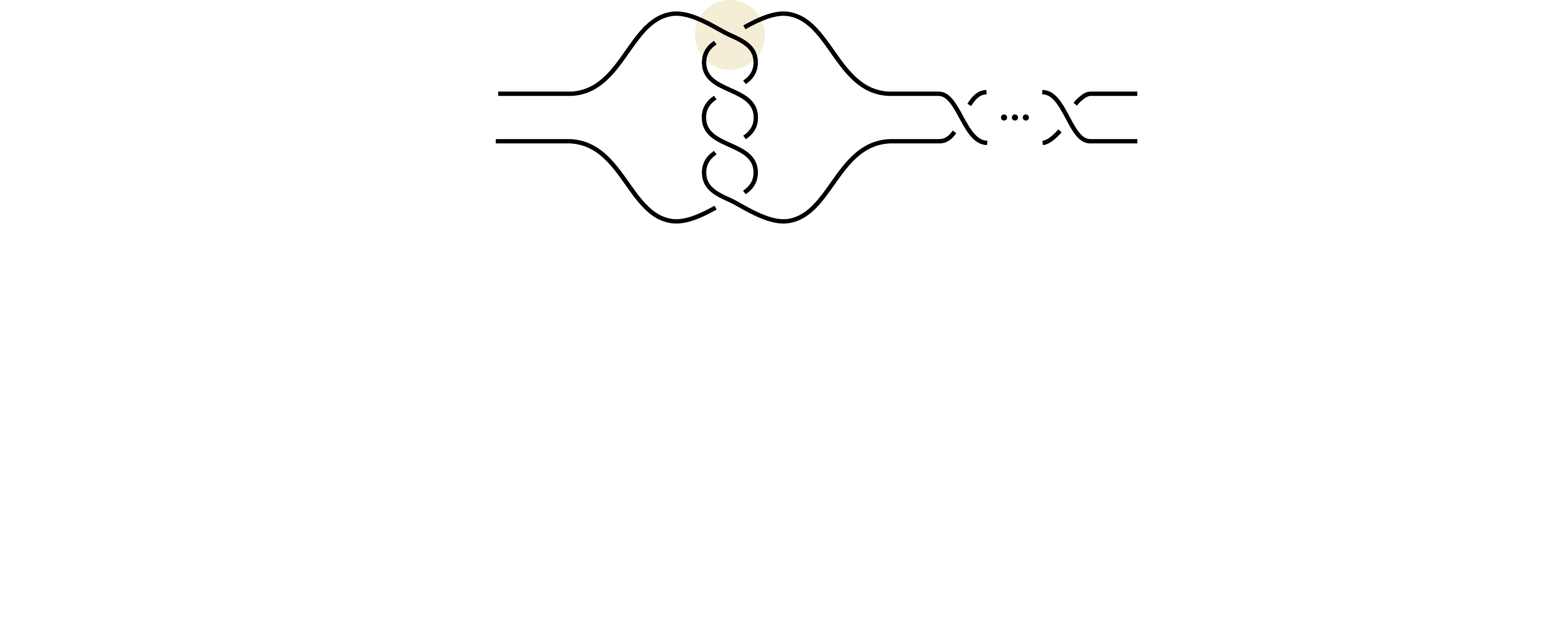
    \caption{The local skein triple for Lemma~\ref{lemma:double_cable}, depicted for $k=2$.} 
    \label{figure:skein_double}
\end{figure} 
To see the role this plays, the reader may look ahead to the decomposition of $K_B$ depicted in Figures~\ref{figure:big_skein_1}--\ref{figure:big_skein_2}. 

To simplify our calculations,  we now  present a series of lemmas  describing the behaviour of the zeroth coefficient polynomial for such skein triples. 
We also describe the behaviour of the zeroth coefficient polynomial for decompositions of certain links into their components.

\subsubsection{Zeroth coefficient polynomials from skein triples} 
\label{subsubsection:0th_skein}

The following results come from applying the skein relation to variants of the skein triple in Figure~\ref{figure:skein_double}. 
In all cases of interest, we will consider skein triples $(L_+,L_-,L_0)$ in which $|L_\pm| = |L_0|-1$, in which case we have
\begin{equation}
\label{equation:simple_skein}
    \Gamma_{L_+}(\alpha) = -\alpha \big( \Gamma_{L_-}(\alpha)+\Gamma_{L_0}(\alpha)\big).
\end{equation}

\begin{lemma}
\label{lemma:double_cable}
    $\Gamma_{D^k_m(C(R))}(\alpha) = -\alpha \big(\Gamma_{D^{k-1}_m(C(R))}(\alpha) + \Gamma_{\overline{C}_{2m,2}(C(R))}(\alpha)\big)$. 
\end{lemma}

\begin{proof}
    We apply the skein relation to the triple 
    \[(L_+, L_-, L_0) = \big(D^k_m(C(R)), D^{k-1}_m(C(R)), \overline{C}_{2m,2}(C(R))\big),\] 
    which satisfies $|L_\pm|=|L_0|-1=1$. 
\end{proof}

We will later apply this with $k\in\{1,2\}$ and $m=qt$. 

\begin{lemma} 
\label{lemma:double}
    $\Gamma_{[D^k_m]}(\alpha) = -\alpha \big(\Gamma_{[D^{k-1}_m]}(\alpha) + \Gamma_{[\overline{C}_{2m,2}]}(\alpha)\big)$. 
\end{lemma} 

\begin{proof}
    We apply the skein relation to the triple 
    \[(L_+, L_-, L_0) = \big([D^k_m], [D^{k-1}_m], [\overline{C}_{2m,2}]\big),\] 
    which satisfies $|L_\pm|=|L_0|-1=1$. 
\end{proof}

We will later apply this with $k\in\{1,2\}$ and $m=qt$. 

\begin{lemma} 
\label{lemma:double_double}
    $\Gamma_{[D^l_n \circ D^k_m]}(\alpha) = -\alpha \big(\Gamma_{[D^{l-1}_n \circ D^k_m]}(\alpha) + \Gamma_{[\overline{C}_{2n,2} \circ D^k_m]}(\alpha)\big)$. 
\end{lemma} 

\begin{proof}
    We apply the skein relation to the triple from Figure~\ref{figure:skein_iterated_double},
    \[(L_+, L_-, L_0) = \big([D^l_n \circ D^k_m], [D^{l-1}_n \circ D^k_m], [\overline{C}_{2n,2} \circ D^k_m]\big),\] 
    which satisfies $|L_\pm|=|L_0|-1=1$. 
\end{proof}

We will later apply this with $(k,l)\in\{(1,1),(1,2),(2,1)\}$ and $(m,n)=(qt,0)$.

\subsubsection{Zeroth coefficient polynomials from link components} 
\label{subsubsection:0th_link}

We will need some additional lemmas obtained by splitting a link into its components and applying Proposition \ref{proposition:HOMFLYPT_link}. 

\begin{lemma}
\label{lemma:cable_cable} 
    $\Gamma_{\overline{C}_{2m,2}(C(R))}(\alpha) = -(1+\alpha^{-1}) (-\alpha)^{-m} \Gamma_{C(R)}(\alpha)^2$. 
\end{lemma} 

\begin{proof}
    For any knot $J$, the link $\overline{C}_{2m,2}(J)$ consists of two oppositely-oriented copies of $J$ that have linking number $-m$. 
    Applying Proposition \ref{proposition:HOMFLYPT_link} to $\overline{C}_{2m,2}(C(R))$ and using the fact that invertibility implies $\Gamma_{\overline{C}(R)}=\Gamma_{C(R)}$ yields the claim.
\end{proof}

We will apply this with $m=qt$. 

\begin{lemma}
\label{lemma:cable}
    $\Gamma_{[\overline{C}_{2m,2}]}(\alpha) = -(1+\alpha^{-1}) (-\alpha)^{qr-m} \Gamma_{[H]}(\alpha) \Gamma_{C(R)}(\alpha)$. 
\end{lemma}

\begin{proof} 
    To apply Proposition~\ref{proposition:HOMFLYPT_link} to the link $[\overline{C}_{2m,2}]$, we must determine its linking number. 
    The link $[\overline{C}_{2m,2}]=U \#_\beta \, \overline{C}_{2m,2}(\overline{C}(R))$ is formed from two oppositely-oriented copies of $C(R)$ by band-summing one of these copies with $U=\widehat{B}$ to form $[H]$. 
    For clarity, write $\overline{C}_{2m,2}(\overline{C}(R))=\overline{C}(R)\cup C(R)$. 
    Then $[H]=\widehat{B} \#_\beta \, \overline{C}(R)$. 
    Thus the linking number of $[\overline{C}_{2m,2}] = [H] \cup C(R)$ is given by
    \begin{align*}
        lk([\overline{C}_{2m,2}]) &= lk([H], C(R)) \\
        &= lk(\widehat{B}, C(R)) + lk(\overline{C}(R),C(R))\\
        &= lk(\widehat{B},C(R))+ lk(\overline{C}_{2m,2}(C(R)))
        \\
        &= lk(\widehat{B}, C(R))  - m. 
    \end{align*}
    
    To determine the linking number $lk(\widehat{B},C(R))$, first recall that $R$ is an $(r,s)$-cable of $\widehat{B}$, hence has linking number $r$ with $\widehat{B}$. 
    Then $C(R)=C_{t,q}(R)$ is homologous to $q$ copies of $R$ (in the complement of $\widehat{B}$), hence has linking number $qr$ with $\widehat{B}$. 
\end{proof}

We will apply this with $m=qt$. 

\begin{lemma}
\label{lemma:cable_double}
    $\Gamma_{[\overline{C}_{2n,2} \circ D^k_m]}(\alpha) = -(1+\alpha^{-1}) (-\alpha)^{-n} \Gamma_{[D^k_m]}(\alpha) \Gamma_{D^k_m(C(R))}(\alpha)$. 
\end{lemma} 

\begin{proof}
    We argue along the same lines as in the proof of Lemma~\ref{lemma:cable}. 
    Apply Proposition \ref{proposition:HOMFLYPT_link} to $[\overline{C}_{2n,2} \circ D^k_m] = [D^k_m] \cup D^k_m(C(R))$, which has linking number 
    \begin{align*}
       lk([\overline{C}_{2n,2} \circ D^k_m]) &= lk([D^k_m], D^k_m(C(R))) \\
       &= lk(\widehat{B}, D^k_m(C(R))) + lk(\overline{C}_{2n,2}(D^k_m(C(R)))) \\
       &= -n. 
       \qedhere 
    \end{align*}
\end{proof}

We will apply this with $k\in\{1,2\}$ and $(m,n)=(qt,0)$. 

\begin{figure}[htbp!] 
    \centering 
    \def\svgwidth{\linewidth}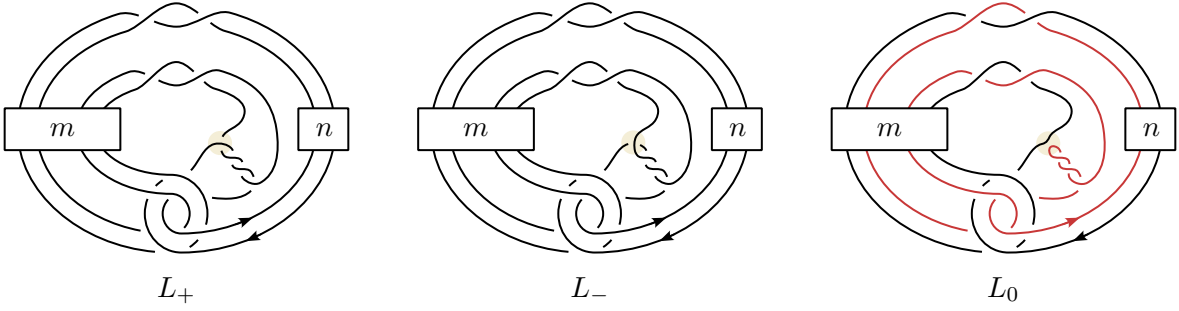
    \caption{The patterns in the skein triple for Lemma~\ref{lemma:double_double}, depicted for $(k,l)=(1,2)$.} 
    \label{figure:skein_iterated_double}
\end{figure}

\subsection{Computing the polynomials} 
\label{subsection:polynomial_computation}

We now express the zeroth coefficient polynomials for $K_B$ and $K_G$ in terms of the polynomials of the simpler knots $C(R)$ and $[H]$ discussed above. 

We will encode the decompositions of $K_B$ and $K_G$ in \emph{skein/linking trees} (Figures~\ref{figure:skein_tree_K_B}--\ref{figure:skein_tree_K_G}), using the following conventions to differentiate between skein and linking relations. 
\begin{itemize}
    \item[\raisebox{2pt}{\rule{17pt}{1pt}}] Solid lines indicate skein relations: we take the sum of the two terms in the next generation and multiply by $-\alpha$, as in the skein relation \eqref{equation:simple_skein}.
    \item[\raisebox{2pt}{\rule{3pt}{1pt}\hspace{2pt}\rule{3pt}{1pt}\hspace{2pt}\rule{3pt}{1pt}\hspace{2pt}\rule{3pt}{1pt}}] Dashed lines indicate linking relations: we take the product of the two terms in the next generation and multiply by $-(1+\alpha^{-1})(-\alpha)^{lk(L)}$ (where $lk(L)$ is the boxed integer), as in the linking formula \eqref{equation:linking}. 
\end{itemize}

The next two results use this strategy to describe the zeroth coefficient polynomials $\Gamma_{K_B}$ and $\Gamma_{K_G}$ in terms of $\Gamma_U=1$, $\Gamma_{C(R)}=\Gamma_{\overline{C}(R)}$ and $\Gamma_{[H]}$. 

\begin{proposition}
\label{proposition:K_B}
    The zeroth coefficient polynomial of $K_B$ is 
    \[\Gamma_{K_B} = -\alpha + (\alpha+1) \big( \alpha^2 - (\alpha^2-1)(-\alpha)^{q(r-t)} \, \Gamma_{[H]} \Gamma_{C(R)} \big) \big( \alpha^2-(\alpha^2-1) (-\alpha)^{-qt}  \big(\Gamma_{C(R)}\big)^2 \big).\] 
\end{proposition} 

\begin{proof}
    The diagrams in Figures~\ref{figure:big_skein_1}--\ref{figure:big_skein_2} break down $K_B = [D^1 \circ D^2_{qt}]$ through skein and linking relations, resulting in a tree whose branches all end with either $U$, $C(R)$, or $[H]$. 
    This is then expressed more compactly Figure \ref{figure:skein_tree_K_B}, which also records linking numbers where relevant. 
    (These linking numbers were justified within the proofs of the recent series of lemmas.) 
    Applying the skein and linking relations will yield the given expression for the zeroth coefficient polynomial of $K_B$. 

\begin{figure}[htbp!] 
\vspace{2em}
    \centering 
    \def\svgwidth{\linewidth}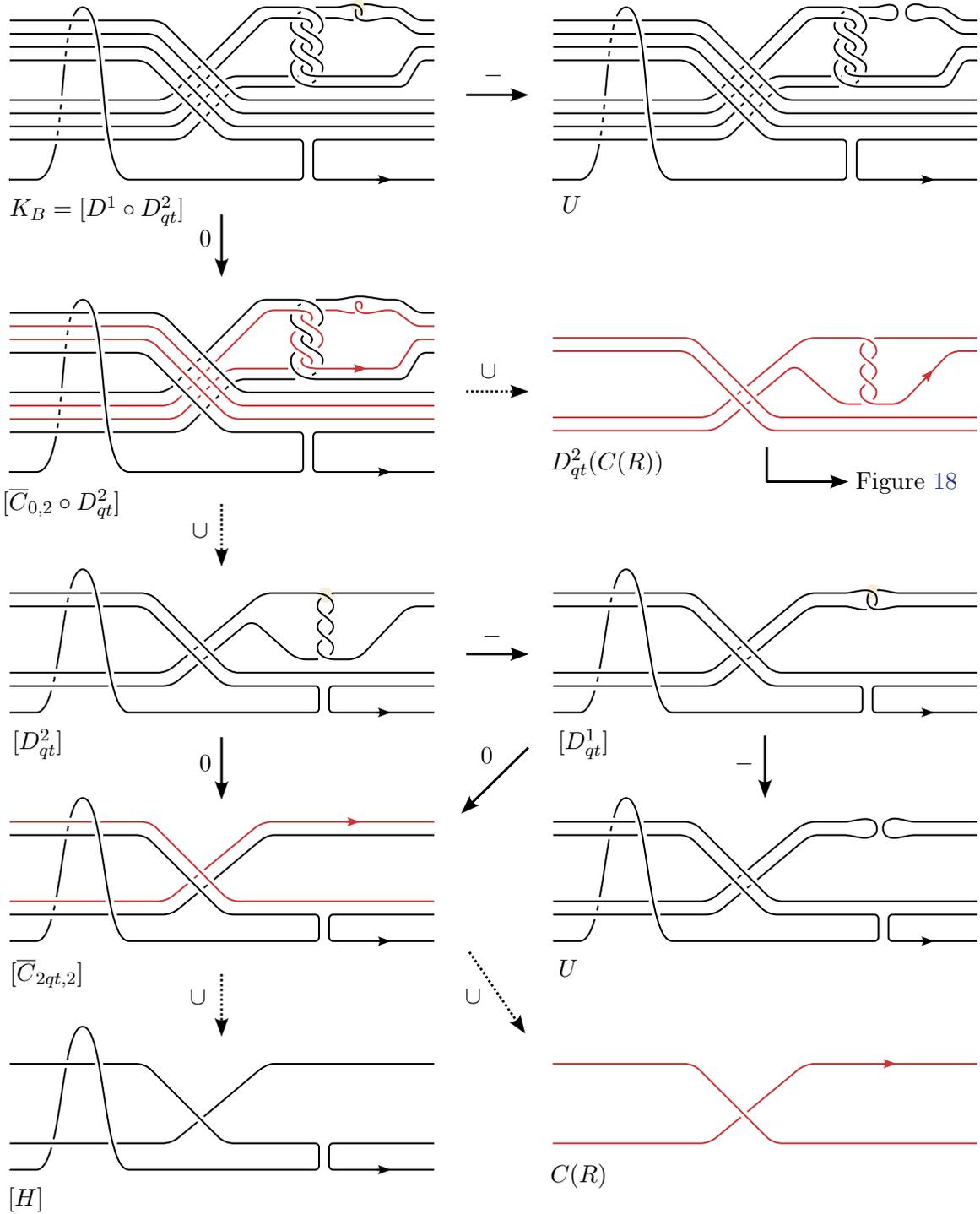
    \caption{A decomposition of $K_B$.} 
    \label{figure:big_skein_1}
\end{figure} 

\begin{figure}[htbp!] 
    \centering 
    \def\svgwidth{.65\linewidth}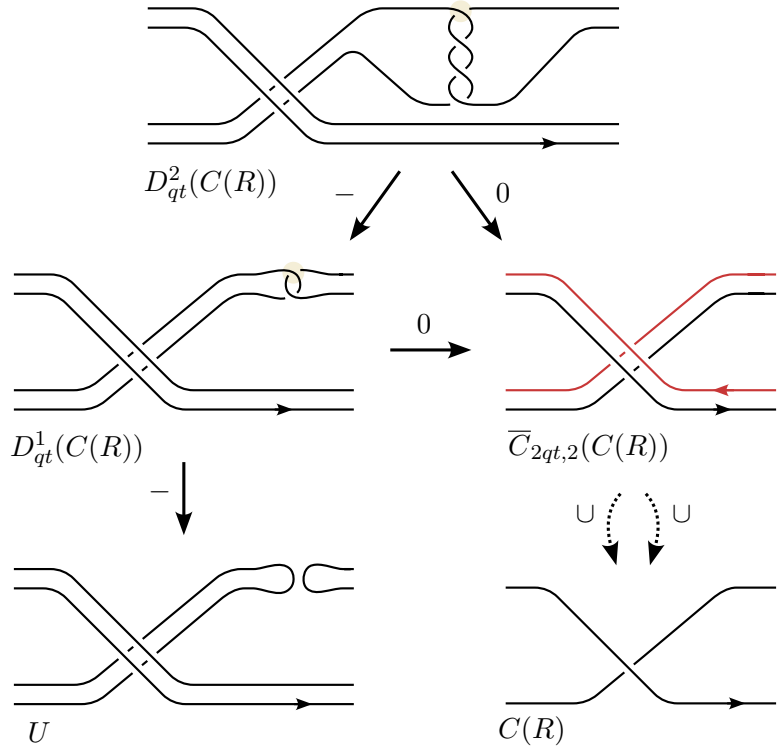
    \caption{A further decomposition of one of the stages of Figure \ref{figure:big_skein_1}.} 
    \label{figure:big_skein_2}
\end{figure} 

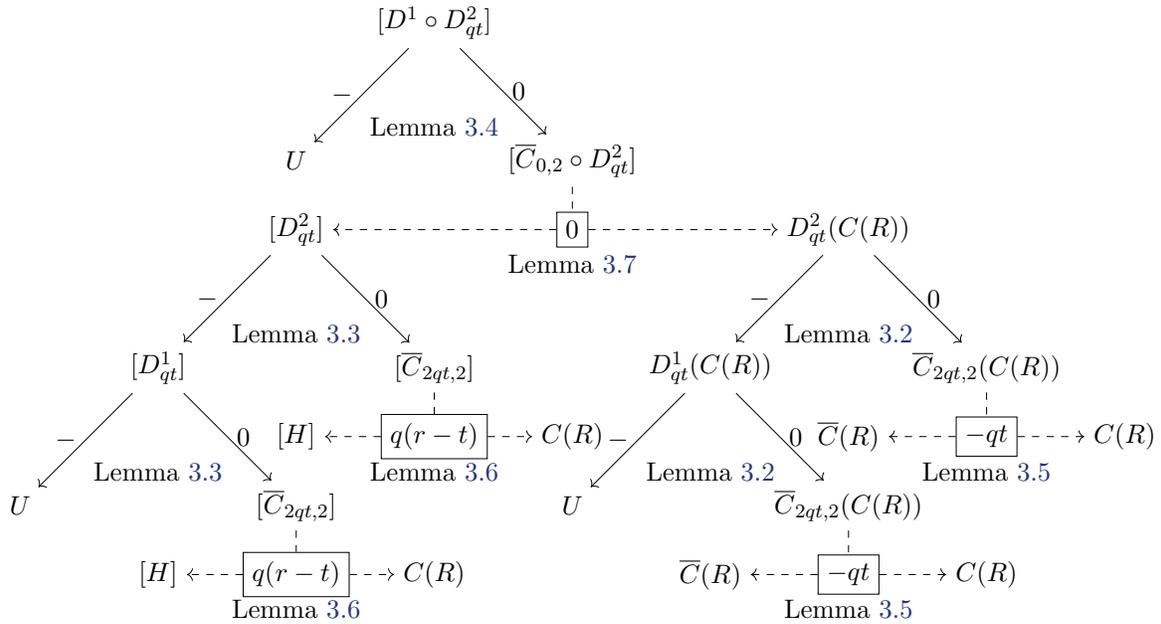
\begin{figure}[htbp!]
\vspace{3em}
    \centering 
    \resizebox{\textwidth}{!}{\input{skein_tree_K_B}} 
    \caption{The full skein/linking tree for $K_B$.} 
    \label{figure:skein_tree_K_B} 
\end{figure} 
    
    To calculate the zeroth coefficient polynomial for $K_B$, we work backwards from the outermost branches of the tree in Figure~\ref{figure:skein_tree_K_B}, expressing the zeroth coefficient polynomials at each stage in terms of those from the previous stages. 
    
    The polynomials occurring in the subtree rooted at $[D^2_{qt}]$ are related as follows. 
        \begin{align*}
            \Gamma_{[\overline{C}_{2qt,2}]} & \overset{\eqref{equation:linking}}{=} - (1+\alpha^{-1}) (-\alpha)^{q(r-t)} \,\Gamma_{[H]}\, \Gamma_{C(R)}
            \\
            &=-\alpha^{-1} (\alpha+1) (-\alpha)^{q(r-t)} \,\Gamma_{[H]}\, \Gamma_{C(R)}
            \\[10pt]
            \Gamma_{[D^1_{qt}]} &\overset{\eqref{equation:simple_skein}}{=}  -\alpha\big( \Gamma_U + \Gamma_{[\overline{C}_{2qt,2}]}\big)
            \\
            &=-\alpha \big(1-\alpha^{-1} (\alpha+1) (-\alpha)^{q(r-t)} \,\Gamma_{[H]}\, \Gamma_{C(R)}\big)
            \\
            &= -\alpha + (\alpha+1)(-\alpha)^{q(r-t)} \,  \Gamma_{[H]} \Gamma_{C(R)}
            \\[10pt]
            \Gamma_{[D^2_{qt}]} &\overset{\eqref{equation:simple_skein}}{=} -\alpha\big( \Gamma_{[D^1_{qt}]} +\Gamma_{[\overline{C}_{2qt,2}]}\big)
            \\
            &= -\alpha\big(-\alpha+(\alpha+1)(-\alpha)^{q(r-t)} \,\Gamma_{[H]}\Gamma_{C(R)} - \alpha^{-1} (\alpha+1) (-\alpha)^{q(r-t)} \,\Gamma_{[H]}\, \Gamma_{C(R)} \big)
            \\
            &= -\alpha\big( -\alpha+(1-\alpha^{-1})(\alpha+1) (-\alpha)^{q(r-t)} \,\Gamma_{[H]}\Gamma_{C(R)}\big)
            \\
            &=\alpha^2 - (\alpha^2-1)(-\alpha)^{q(r-t)} \, \Gamma_{[H]} \Gamma_{C(R)}
        \end{align*}
        
    The polynomials in the subtree rooted at $D^2_{qt}(C(R))$ are related similarly. 
        \begin{align*}
            \Gamma_{\overline{C}_{2qt,2}(C(R))}&\overset{\eqref{equation:linking}}{=} -(1+\alpha^{-1})(-\alpha)^{-qt} \big(\Gamma_{C(R)}\big)^2 
            \\
            &= -\alpha^{-1}(\alpha+1)(-\alpha)^{-qt}\big(\Gamma_{C(R)}\big)^2
            \\[10pt]
            \Gamma_{D^1_{qt}(C(R))} &\overset{\eqref{equation:simple_skein}}{=} -\alpha\big(\Gamma_{U}+\Gamma_{\overline{C}_{2qt,2}(C(R))}\big)
            \\
            &= -\alpha \big(1-\alpha^{-1}(\alpha+1)(-\alpha)^{-qt} \big(\Gamma_{C(R)}\big)^2 \big)
            \\
            &= -\alpha+(\alpha+1)(-\alpha)^{-qt}\big(\Gamma_{C(R)}\big)^2
            \\[10pt]
            \Gamma_{D^2_{qt}(C(R))} &\overset{\eqref{equation:simple_skein}}{=} -\alpha \big( \Gamma_{D^1_{qt}(C(R))} +\Gamma_{\overline{C}_{2qt,2}(C(R))}\big)
            \\
            &= -\alpha\big(-\alpha+(\alpha+1)(-\alpha)^{-qt}\big(\Gamma_{C(R)}\big)^2 -\alpha^{-1}(\alpha+1)(-\alpha)^{-qt}\big(\Gamma_{C(R)}\big)^2 \big)
            \\
            &= -\alpha\big(-\alpha+(1-\alpha^{-1})(\alpha+1) (-\alpha)^{-qt}  \big(\Gamma_{C(R)}\big)^2\big)
            \\
            &= \alpha^2-(\alpha^2-1) (-\alpha)^{-qt}  \big(\Gamma_{C(R)}\big)^2
        \end{align*}
        
    We can now connect both halves of the subtree rooted at $[\overline{C}_{0,2} \circ D^2_{qt}]$ to give the following. 
        \begin{align*}
            \Gamma_{[\overline{C}_{0,2} \circ D^2_{qt}]} &\overset{\eqref{equation:linking}}{=} -(1+\alpha^{-1})(-\alpha)^0 \, \Gamma_{[D^2_{qt}]} \Gamma_{D^2_{qt}(C(R))}
            \\
            &= -\alpha^{-1}(\alpha+1) \big(\alpha^2 - \big(\alpha^2-1\big)(-\alpha)^{q(r-t)} \, \Gamma_{[H]} \Gamma_{C(R)}\big)\big(\alpha^2-(\alpha^2-1) (-\alpha)^{-qt}  \big(\Gamma_{C(R)}\big)^2\big)
        \end{align*}
    This yields the desired expression for $\Gamma_{K_B}$. 
        \begin{align*}
            \Gamma_{K_B} &\overset{\eqref{equation:simple_skein}}{=} -\alpha\big( \Gamma_{U} + \Gamma_{[\overline{C}_{0,2} \circ D^2_{qt}]}\big)
            \\
            &= -\alpha\big(1-\alpha^{-1}(\alpha+1) \big(\alpha^2 - (\alpha^2-1)(-\alpha)^{q(r-t)} \, \Gamma_{[H]} \Gamma_{C(R)}\big)\big(\alpha^2-(\alpha^2-1) (-\alpha)^{-qt}  \big(\Gamma_{C(R)}\big)^2\big)\big)
            \\
            &= -\alpha+(\alpha+1) \big(\alpha^2 - (\alpha^2-1)(-\alpha)^{q(r-t)} \, \Gamma_{[H]} \Gamma_{C(R)}\big)\big(\alpha^2-(\alpha^2-1) (-\alpha)^{-qt}  \big(\Gamma_{C(R)}\big)^2\big)\qedhere
        \end{align*}
\end{proof}

\begin{proposition}
\label{proposition:K_G}
    The zeroth coefficient polynomial of $K_G$ is 
    \[\Gamma_{K_G} = \alpha^2 - (\alpha^2-1) \big( \alpha - (\alpha+1) (-\alpha)^{q(r-t)} \Gamma_{[H]} \Gamma_{C(R)} \big) \big( \alpha - (\alpha+1) (-\alpha)^{-qt} \big(\Gamma_{C(R)}\big)^2 \big).\] 
\end{proposition}

\begin{proof}
    This follows along the same lines as Proposition~\ref{proposition:K_B}, applying skein/linking relations to $K_G=[D^2 \circ D^1_{qt}]$ instead of $K_B=[D^1 \circ D^2_{qt}]$. 
    We skip directly to the corresponding skein/linking tree in Figure~\ref{figure:skein_tree_K_G}. 
    The only new linking number that did not appear in Figure~\ref{figure:skein_tree_K_B} is
    $$lk\left([\overline{C}_{0,2} \circ D^1_{qt}]\right)=0,$$
    which can be justified as done above for $lk[\overline{C}_{0,2} \circ D^2_{qt}]$.

    The zeroth coefficient polynomials for most stages of the tree were previously computed in Proposition~\ref{proposition:K_B}. 
    Given the similarity of these two calculations, we leave the rest to the reader. 
\end{proof} 
    
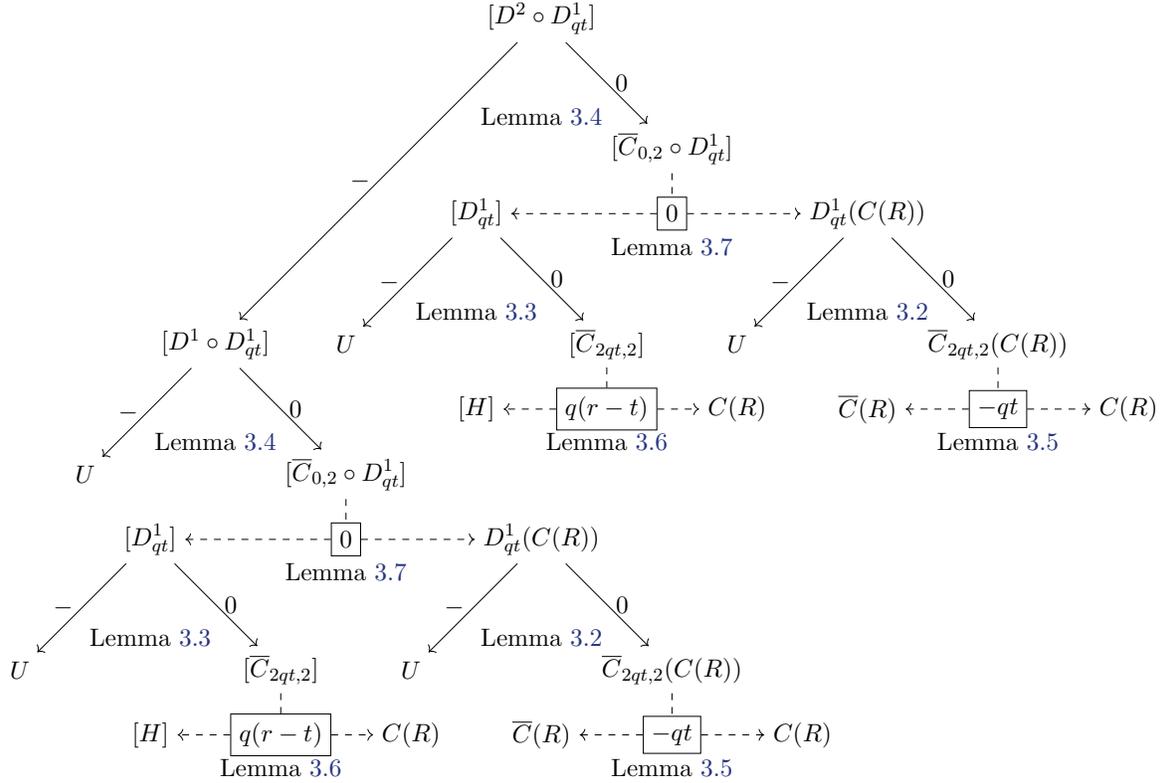
\begin{figure}[htbp!]
    \centering 
    \resizebox{\textwidth}{!}{\input{skein_tree_K_G}} 
    \caption{The full skein/linking tree for $K_G$.} 
    \label{figure:skein_tree_K_G} 
\end{figure}

\subsection{Distinguishing the polynomials} 
\label{subsection:0th_+ve_braid}

It is difficult to evaluate the polynomials $\Gamma_{C(R)}$ and $\Gamma_{[H]}$ which appear in our formulae, but in fact it will be fine to continue working in terms of these polynomials. 
The following lemma gives a sufficient condition for $\Gamma_{K_B} \neq \Gamma_{K_G}$. 

\begin{lemma}
\label{lemma:factorisation}
    If the zeroth coefficient polynomial $\Gamma_{C(R)}$ of $C(R)$ is not a unit in $\mathbb{Z}[\alpha]$, then $\Gamma_{K_B} \neq \Gamma_{K_G}$. 
\end{lemma} 

\begin{proof}
    We start by computing the difference between the polynomials $\Gamma_{K_B}$ and $\Gamma_{K_G}$ which is 
    \begin{align*}
        \alpha (1+\alpha)^2 (1-\alpha^2) \big( 1 - (-\alpha)^{q(r-t)} \Gamma_{[H]}(\alpha) \Gamma_{C(R)}(\alpha) \big) \big( 1 - (-\alpha)^{-qt }\Gamma_{C(R)}(\alpha)^2 \big). 
    \end{align*}
    To obstruct this difference from being identically the zero polynomial, we need both 
    \begin{align*}
        \Gamma_{[H]}(\alpha)\Gamma_{C(R)}(\alpha) &\neq (-\alpha)^{q(t-r)}, \\
        \Gamma_{C(R)}(\alpha)^2 &\neq (-\alpha)^{qt}. 
    \end{align*}
    It is sufficient to obstruct $\Gamma_{C(R)}(\alpha)$ from being a unit in $\mathbb{Z}[\alpha^{\pm1}]$. 
\end{proof}

It will now be helpful to choose our parameters more carefully. 
We can assume that $p/q\geq0$, since the $p/q$-framed Dehn surgery function is injective if and only the $-p/q$-framed Dehn surgery function is injective. 
Moreover, the cases $p=0$ and $p=1$ have already been solved by other methods \cite{Lickorish,Brakes,Kawauchi}, so we can take $p>1$ in our construction. 

We say a knot is a \emph{positive braid knot} if it arises as the closure of a positive braid. 
For a braid $\beta$, we use $\widehat{\beta}$ to denote its closure. 
We will show that, for $p>1$, the knot $C(R)$ can be chosen to be a positive braid knot and thus its zeroth coefficient polynomial $\Gamma_{C(R)}$ is not a unit, after which it will follow from Lemma~\ref{lemma:factorisation} that $\Gamma_{K_B} \neq \Gamma_{K_G}$. 

\begin{lemma}
\label{lemma:positive_braid_knot}
    If $p>1$, then $C(R)$ can be chosen to be a positive braid knot. 
\end{lemma}

\begin{proof}
    Fix $p/q$ and recall that $C(R) = C_{t,q}(T_{r,s})$ where $t=-s(1-qr)$ for some integers $r,s$ satisfying $ps-qr=1$. 
    Note that, by Bézout's lemma, there are infinitely many choices for $r,s$ and we can choose $s\geq1$.
    Since $p,q$ and $s$ are all non-negative, $r$ too must be non-negative.

    Now consider the diagram for $C(R)$ in which $R=T_{r,s}$ is drawn as a positive braid with $r$ right-handed $1/s$-twists (i.e.~the $s^{\text{th}}$ root of a full twist) on $s$ strands and the cable operator $C=C_{t,q}$ with winding number $q$ is applied to $R$. 
    To construct a braid diagram for the cable, we begin with $q$ copies of $R$ as blackboard pushoffs, which form a positive braid. 
    To compensate for the writhe $w$ of $R$, we introduce $w=r(s-1)$ full left-handed twists into the $q$ strands. 
    Finally, to produce the desired cable, we introduce $t$ right-handed $1/q$-twists (i.e.~the $q^{\text{th}}$ root of a full twist) into the $q$ strands. 
    Under the additional assumption that $p \geq 2$, the net number of right-handed $1/q$-twists that have been added is 
    \[t-qw=-s(1-qr) - qr(s-1) = qr-s = (p-1)s - 1 \geq s-1 \geq 0,\] 
    which implies that $C(R)$ is a positive braid knot.
\end{proof}

To prove Theorem \ref{theorem:non-characterising}, we combine Lemma \ref{lemma:factorisation} with the following. 

\begin{proposition}
\label{proposition:positive_braid_knot}
    If K is a non-trivial positive braid knot, then its zeroth coefficient polynomial $\Gamma_{K}(\alpha)$ is not a unit in $\mathbb{Z}[\alpha^{\pm1}]$.     
\end{proposition}

The key input in the proof of Proposition \ref{proposition:positive_braid_knot} is a result of \cite[Theorem~2]{Ito}, which shows that if $L$ is the closure of a positive braid, then its \emph{normalised} HOMFLYPT polynomial $\widetilde{P}_L(\alpha,z)$ is positive, in the sense that all coefficients are non-negative. 
For this to be useful to us, we need to understand Ito's normalisation and use a formula for the \emph{normalised} zeroth coefficent polynomial, which inherits the positivity property. 
Notice, however, that positivity alone is not enough to prove a polynomial is not a unit, so we will then have to do something a little sneaky.

We begin with the definitions. 
The \emph{normalised HOMFLYPT polynomial} of a link $L$ is defined in \cite[Definition~1]{Ito} as
        \[\widetilde{P}_L(\alpha,z) := (\alpha+1)^{1-s(L)} (-\alpha)^{\tfrac{\chi(L)-2+|L|}{2}} \left.\left(\frac{z}{v}\right)^{|L|-1} P_L(v,z)\right\vert_{\alpha=-v^2}\] 
        where: 
        \begin{itemize}
            \item $s(L)$ is the number of non-split link components of $L$; 
            \item $\chi(L)$ is the maximal Euler characteristic of a Seifert surface for $L$. 
        \end{itemize} 

        Writing $\Gamma^i_L(\alpha) = p^i_L(v)|_{\alpha=-v^2}$ as before and using the definition of the HOMFLYPT polynomial, we obtain
        \[\sum_{i\geq0}\Gamma^i_L(\alpha)z^{2i} = (\alpha+1)^{s(L)-1} (-\alpha)^{\tfrac{-\chi(L)+2-|L|}{2}} \widetilde{P}_L(\alpha,z).\]
        Setting $z=0$, this yields the following formula for the zeroth coefficient polynomial: 
        \begin{equation}
        \label{equation:zeroth}
            \Gamma_L(\alpha) = (\alpha+1)^{s(L)-1} (-\alpha)^{\tfrac{-\chi(L)+2-|L|}{2}} \widetilde{\Gamma}_L(\alpha)
        \end{equation}
        where $\widetilde{\Gamma}_L(\alpha) := \widetilde{P}_L(\alpha,0)$ is the \emph{normalised zeroth coefficient polynomial}. 

\begin{proof}[Proof of Proposition \ref{proposition:positive_braid_knot}]
        We know from the above discussion that for such a knot, the normalised zeroth coefficient polynomial is positive. 
        We will show that the normalised zeroth coefficient polynomial for our knot is in fact a sum of positive polynomials. 
        To do so, we introduce yet another skein triangle. 

        By \cite[Lemma 2]{Van_Buskirk}, any nontrivial positive braid knot $L=L_+$ admits a skein triple 
        \[(L_+,L_0,L_-) = (\widehat{\beta\sigma^2}, \widehat{\beta\sigma}, \widehat{\beta})\] 
        consisting only of positive braid links. 
        In our setting, $L_+=K$ is a positive braid knot, so we have $|L_\pm|=1$; moreover, as it is impossible for merging of link components to occur, we must have $|L_0|=2$. 
        Hence the skein relation for the zeroth coefficient polynomial becomes
        \[\Gamma_{L_+}(\alpha) = -\alpha (\Gamma_{L_-}(\alpha)+\Gamma_{L_0}(\alpha)).\] 
        
        Since $L_\pm$ are knots, we have $s(L_\pm)=1$. 
        For any positive braid link $L=\widehat{\beta}$, Bennequin's inequality implies that $\chi(L) = n-e(\beta)$, where $\beta \in B_n$ and $e(\beta)$ is the sum of the exponents in the braid word \cite[Subsection~2.1]{Ito}. 
        Using this fact, the exponent of $-\alpha$ in Equation \ref{equation:zeroth} applied to each of the links in the skein triple can be rewritten in terms of $g(K)$. 
        We leave it to the reader to derive from the skein relation the following formula \cite[Equation~(1)]{Ito}: 
        \[\widetilde{\Gamma}_{L_+}(\alpha) = \widetilde{\Gamma}_{L_-}(\alpha) + (1+\alpha)^{s(L_0)-1}\widetilde{\Gamma}_{L_0}(\alpha).\] 

        Observe that we can apply Ito's result to each term in the right hand side; in particular, we have a sum of non-zero polynomials with non-negative coefficients, which prevents cancellation of terms. 
        This ensures that $\widetilde{\Gamma}_K(\alpha)$, and thus $\Gamma_K(\alpha)$, is not a unit in $\mathbb{Z}[\alpha^{\pm1}]$. 
    \end{proof}

This completes the proof of Theorem \ref{theorem:non-characterising}.

\bigskip

\bibliographystyle{amsalpha}
\bibliography{bibliography.bib}

\end{document}

%% file: K_P_mod-mod.pdf_tex
%% Creator: Inkscape 1.2 (dc2aeda, 2022-05-15), www.inkscape.org
%% Accompanies image file 'K_P_mod-mod.pdf' (pdf, eps, ps)
\begingroup%
  \makeatletter%
  \providecommand\color[2][]{%
    \errmessage{(Inkscape) Color is used for the text in Inkscape, but the package 'color.sty' is not loaded}%
    \renewcommand\color[2][]{}%
  }%
  \providecommand\transparent[1]{%
    \errmessage{(Inkscape) Transparency is used (non-zero) for the text in Inkscape, but the package 'transparent.sty' is not loaded}%
    \renewcommand\transparent[1]{}%
  }%
  \providecommand\rotatebox[2]{#2}%
  \newcommand*\fsize{\dimexpr\f@size pt\relax}%
  \newcommand*\lineheight[1]{\fontsize{\fsize}{#1\fsize}\selectfont}%
  \ifx\svgwidth\undefined%
    \setlength{\unitlength}{657.95419287bp}%
    \ifx\svgscale\undefined%
      \relax%
    \else%
      \setlength{\unitlength}{\unitlength * \real{\svgscale}}%
    \fi%
  \else%
    \setlength{\unitlength}{\svgwidth}%
  \fi%
  \global\let\svgwidth\undefined%
  \global\let\svgscale\undefined%
  \makeatother%
  \begin{picture}(1,0.34626603)%
    \lineheight{1}%
    \setlength\tabcolsep{0pt}%
    \put(0,0){\includegraphics[width=\unitlength,page=1]{K_P_mod-mod.pdf}}%
    \put(0.01272826,0.07934204){\color[rgb]{0.6,0.6,0.6}\makebox(0,0)[t]{\smash{\begin{tabular}[t]{c}{\small $U$}\end{tabular}}}}%
    \put(0,0){\includegraphics[width=\unitlength,page=2]{K_P_mod-mod.pdf}}%
    \put(0.01750631,0.35697404){\color[rgb]{0.6,0.6,0.6}\makebox(0,0)[t]{\smash{\begin{tabular}[t]{c}{\small $C_{t,q}$}\end{tabular}}}}%
    \put(-0.02422248,0.22463322){\color[rgb]{0.6,0.6,0.6}\makebox(0,0)[t]{\smash{\begin{tabular}[t]{c}{\small $C_{r,s}$}\end{tabular}}}}%
    \put(0,0){\includegraphics[width=\unitlength,page=3]{K_P_mod-mod.pdf}}%
    \put(0.0618143,0.03636148){\color[rgb]{0,0,0}\makebox(0,0)[t]{\smash{\begin{tabular}[t]{c}$[P]$\end{tabular}}}}%
    \put(0,0){\includegraphics[width=\unitlength,page=4]{K_P_mod-mod.pdf}}%
    \put(0.14498887,0.22387965){\color[rgb]{0,0,0}\makebox(0,0)[t]{\smash{\begin{tabular}[t]{c}$P$\end{tabular}}}}%
    \put(0,0){\includegraphics[width=\unitlength,page=5]{K_P_mod-mod.pdf}}%
    \put(0.73073183,0.28901249){\color[rgb]{0.22352941,0.22352941,0.6745098}\makebox(0,0)[t]{\smash{\begin{tabular}[t]{c}{\footnotesize $qt$}\end{tabular}}}}%
    \put(0.90509574,0.24765705){\color[rgb]{0.22352941,0.22352941,0.6745098}\makebox(0,0)[t]{\smash{\begin{tabular}[t]{c}{\footnotesize $-4$}\end{tabular}}}}%
    \put(0,0){\includegraphics[width=\unitlength,page=6]{K_P_mod-mod.pdf}}%
    \put(0.73073183,0.12505989){\color[rgb]{0.21960784,0.59607843,0.14509804}\makebox(0,0)[t]{\smash{\begin{tabular}[t]{c}{\footnotesize $qt$}\end{tabular}}}}%
    \put(0,0){\includegraphics[width=\unitlength,page=7]{K_P_mod-mod.pdf}}%
    \put(0.89753056,0.08370446){\color[rgb]{0.21960784,0.59607843,0.14509804}\makebox(0,0)[t]{\smash{\begin{tabular}[t]{c}{\footnotesize $-2$}\end{tabular}}}}%
    \put(0,0){\includegraphics[width=\unitlength,page=8]{K_P_mod-mod.pdf}}%
    \put(0.70793392,0.22517746){\color[rgb]{0.22352941,0.22352941,0.6745098}\makebox(0,0)[t]{\smash{\begin{tabular}[t]{c}{$P'_B$}\end{tabular}}}}%
    \put(0.70793392,0.06559042){\color[rgb]{0.21960784,0.59607843,0.14509804}\makebox(0,0)[t]{\smash{\begin{tabular}[t]{c}{$P'_G$}\end{tabular}}}}%
    \put(0,0){\includegraphics[width=\unitlength,page=9]{K_P_mod-mod.pdf}}%
  \end{picture}%
\endgroup%

%% file: initial-rbg-smooth.pdf_tex
%% Creator: Inkscape 1.2 (dc2aeda, 2022-05-15), www.inkscape.org
%% Accompanies image file 'initial-rbg-smooth.pdf' (pdf, eps, ps)
\begingroup%
  \makeatletter%
  \providecommand\color[2][]{%
    \errmessage{(Inkscape) Color is used for the text in Inkscape, but the package 'color.sty' is not loaded}%
    \renewcommand\color[2][]{}%
  }%
  \providecommand\transparent[1]{%
    \errmessage{(Inkscape) Transparency is used (non-zero) for the text in Inkscape, but the package 'transparent.sty' is not loaded}%
    \renewcommand\transparent[1]{}%
  }%
  \providecommand\rotatebox[2]{#2}%
  \newcommand*\fsize{\dimexpr\f@size pt\relax}%
  \newcommand*\lineheight[1]{\fontsize{\fsize}{#1\fsize}\selectfont}%
  \ifx\svgwidth\undefined%
    \setlength{\unitlength}{341.25029556bp}%
    \ifx\svgscale\undefined%
      \relax%
    \else%
      \setlength{\unitlength}{\unitlength * \real{\svgscale}}%
    \fi%
  \else%
    \setlength{\unitlength}{\svgwidth}%
  \fi%
  \global\let\svgwidth\undefined%
  \global\let\svgscale\undefined%
  \makeatother%
  \begin{picture}(1,0.53168104)%
    \lineheight{1}%
    \setlength\tabcolsep{0pt}%
    \put(0,0){\includegraphics[width=\unitlength,page=1]{initial-rbg-smooth.pdf}}%
    \put(1.06533315,0.11362189){\color[rgb]{0.22352941,0.22352941,0.6745098}\makebox(0,0)[t]{\smash{\begin{tabular}[t]{c}{\small $\displaystyle{p}/{q}$}\end{tabular}}}}%
    \put(1.15217322,0.46526926){\color[rgb]{0.78431373,0.21568627,0.21568627}\makebox(0,0)[t]{\smash{\begin{tabular}[t]{c}{\small $\displaystyle{-s(1-qr)}/{q}$}\end{tabular}}}}%
    \put(1.1521732,0.26747027){\color[rgb]{0.21960784,0.59607843,0.14509804}\makebox(0,0)[t]{\smash{\begin{tabular}[t]{c}{\small $ \displaystyle{p(1+q^2 r^2)}/{q}$}\end{tabular}}}}%
    \put(-0.04395603,0.25966251){\color[rgb]{0.21960784,0.59607843,0.14509804}\makebox(0,0)[t]{\smash{\begin{tabular}[t]{c}$\widehat{G}$\end{tabular}}}}%
    \put(-0.04395604,0.46185716){\color[rgb]{0.78431373,0.21568627,0.21568627}\makebox(0,0)[t]{\smash{\begin{tabular}[t]{c}$\widehat{R}$\end{tabular}}}}%
    \put(-0.04395604,0.11020979){\color[rgb]{0.22352941,0.22352941,0.6745098}\makebox(0,0)[t]{\smash{\begin{tabular}[t]{c}$\widehat{B}$\end{tabular}}}}%
    \put(0,0){\includegraphics[width=\unitlength,page=2]{initial-rbg-smooth.pdf}}%
  \end{picture}%
\endgroup%

%% file: rbg-modified-star-new.pdf_tex
%% Creator: Inkscape 1.2 (dc2aeda, 2022-05-15), www.inkscape.org
%% Accompanies image file 'rbg-modified-star-new.pdf' (pdf, eps, ps)
\begingroup%
  \makeatletter%
  \providecommand\color[2][]{%
    \errmessage{(Inkscape) Color is used for the text in Inkscape, but the package 'color.sty' is not loaded}%
    \renewcommand\color[2][]{}%
  }%
  \providecommand\transparent[1]{%
    \errmessage{(Inkscape) Transparency is used (non-zero) for the text in Inkscape, but the package 'transparent.sty' is not loaded}%
    \renewcommand\transparent[1]{}%
  }%
  \providecommand\rotatebox[2]{#2}%
  \newcommand*\fsize{\dimexpr\f@size pt\relax}%
  \newcommand*\lineheight[1]{\fontsize{\fsize}{#1\fsize}\selectfont}%
  \ifx\svgwidth\undefined%
    \setlength{\unitlength}{706.5000519bp}%
    \ifx\svgscale\undefined%
      \relax%
    \else%
      \setlength{\unitlength}{\unitlength * \real{\svgscale}}%
    \fi%
  \else%
    \setlength{\unitlength}{\svgwidth}%
  \fi%
  \global\let\svgwidth\undefined%
  \global\let\svgscale\undefined%
  \makeatother%
  \begin{picture}(1,0.25641824)%
    \lineheight{1}%
    \setlength\tabcolsep{0pt}%
    \put(0,0){\includegraphics[width=\unitlength,page=1]{rbg-modified-star-new.pdf}}%
    \put(-0.02441528,0.12308977){\color[rgb]{0.21960784,0.59607843,0.14509804}\makebox(0,0)[t]{\smash{\begin{tabular}[t]{c}${G}$\end{tabular}}}}%
    \put(-0.02441528,0.22075287){\color[rgb]{0.78431373,0.21568627,0.21568627}\makebox(0,0)[t]{\smash{\begin{tabular}[t]{c}${R}$\end{tabular}}}}%
    \put(-0.02441528,-0.00429924){\color[rgb]{0.22352941,0.22352941,0.6745098}\makebox(0,0)[t]{\smash{\begin{tabular}[t]{c}${B}$\end{tabular}}}}%
    \put(0.11266809,0.04917978){\color[rgb]{0.22352941,0.22352941,0.6745098}\makebox(0,0)[t]{\smash{\begin{tabular}[t]{c}{\footnotesize$P_B$}\end{tabular}}}}%
    \put(0.50000056,0.10049449){\color[rgb]{0,0,0}\makebox(0,0)[t]{\smash{\begin{tabular}[t]{c}$\cong$\end{tabular}}}}%
    \put(0,0){\includegraphics[width=\unitlength,page=2]{rbg-modified-star-new.pdf}}%
    \put(0.60191319,0.06415759){\color[rgb]{0.21960784,0.59607843,0.14509804}\makebox(0,0)[t]{\smash{\begin{tabular}[t]{c}{\footnotesize$P_G$}\end{tabular}}}}%
    \put(0,0){\includegraphics[width=\unitlength,page=3]{rbg-modified-star-new.pdf}}%
  \end{picture}%
\endgroup%

%% file: whitehead-mod.pdf_tex
%% Creator: Inkscape 1.2 (dc2aeda, 2022-05-15), www.inkscape.org
%% Accompanies image file 'whitehead-mod.pdf' (pdf, eps, ps)
\begingroup%
  \makeatletter%
  \providecommand\color[2][]{%
    \errmessage{(Inkscape) Color is used for the text in Inkscape, but the package 'color.sty' is not loaded}%
    \renewcommand\color[2][]{}%
  }%
  \providecommand\transparent[1]{%
    \errmessage{(Inkscape) Transparency is used (non-zero) for the text in Inkscape, but the package 'transparent.sty' is not loaded}%
    \renewcommand\transparent[1]{}%
  }%
  \providecommand\rotatebox[2]{#2}%
  \newcommand*\fsize{\dimexpr\f@size pt\relax}%
  \newcommand*\lineheight[1]{\fontsize{\fsize}{#1\fsize}\selectfont}%
  \ifx\svgwidth\undefined%
    \setlength{\unitlength}{407.26763868bp}%
    \ifx\svgscale\undefined%
      \relax%
    \else%
      \setlength{\unitlength}{\unitlength * \real{\svgscale}}%
    \fi%
  \else%
    \setlength{\unitlength}{\svgwidth}%
  \fi%
  \global\let\svgwidth\undefined%
  \global\let\svgscale\undefined%
  \makeatother%
  \begin{picture}(1,0.90824451)%
    \lineheight{1}%
    \setlength\tabcolsep{0pt}%
    \put(0,0){\includegraphics[width=\unitlength,page=1]{whitehead-mod.pdf}}%
    \put(0.54896578,0.28536472){\color[rgb]{0,0,0}\makebox(0,0)[t]{\smash{\begin{tabular}[t]{c}\footnotesize{$m$}\end{tabular}}}}%
    \put(0,0){\includegraphics[width=\unitlength,page=2]{whitehead-mod.pdf}}%
    \put(0.54041937,0.71755615){\color[rgb]{0,0,0}\makebox(0,0)[t]{\smash{\begin{tabular}[t]{c}{\footnotesize $-k$}\end{tabular}}}}%
    \put(0,0){\includegraphics[width=\unitlength,page=3]{whitehead-mod.pdf}}%
    \put(-0.00226495,0.5732332){\color[rgb]{0.50196078,0.50196078,0.50196078}\makebox(0,0)[t]{\smash{\begin{tabular}[t]{c}\footnotesize{$U$}\end{tabular}}}}%
    \put(0.540413,0.01649944){\color[rgb]{0,0,0}\makebox(0,0)[t]{\smash{\begin{tabular}[t]{c}{$D^k_m$}\end{tabular}}}}%
  \end{picture}%
\endgroup%

%% file: patterns.pdf_tex
%% Creator: Inkscape 1.2 (dc2aeda, 2022-05-15), www.inkscape.org
%% Accompanies image file 'patterns.pdf' (pdf, eps, ps)
\begingroup%
  \makeatletter%
  \providecommand\color[2][]{%
    \errmessage{(Inkscape) Color is used for the text in Inkscape, but the package 'color.sty' is not loaded}%
    \renewcommand\color[2][]{}%
  }%
  \providecommand\transparent[1]{%
    \errmessage{(Inkscape) Transparency is used (non-zero) for the text in Inkscape, but the package 'transparent.sty' is not loaded}%
    \renewcommand\transparent[1]{}%
  }%
  \providecommand\rotatebox[2]{#2}%
  \newcommand*\fsize{\dimexpr\f@size pt\relax}%
  \newcommand*\lineheight[1]{\fontsize{\fsize}{#1\fsize}\selectfont}%
  \ifx\svgwidth\undefined%
    \setlength{\unitlength}{1227.54699226bp}%
    \ifx\svgscale\undefined%
      \relax%
    \else%
      \setlength{\unitlength}{\unitlength * \real{\svgscale}}%
    \fi%
  \else%
    \setlength{\unitlength}{\svgwidth}%
  \fi%
  \global\let\svgwidth\undefined%
  \global\let\svgscale\undefined%
  \makeatother%
  \begin{picture}(1,0.36148574)%
    \lineheight{1}%
    \setlength\tabcolsep{0pt}%
    \put(0,0){\includegraphics[width=\unitlength,page=1]{patterns.pdf}}%
    \put(0.80844755,0.01823679){\color[rgb]{0.21960784,0.59607843,0.14509804}\makebox(0,0)[t]{\smash{\begin{tabular}[t]{c}$P_G$\end{tabular}}}}%
    \put(0.57911109,0.26329111){\color[rgb]{0.22352941,0.22352941,0.6745098}\makebox(0,0)[t]{\smash{\begin{tabular}[t]{c}\footnotesize{$U$}\end{tabular}}}}%
    \put(0,0){\includegraphics[width=\unitlength,page=2]{patterns.pdf}}%
    \put(0.25343256,0.00183183){\color[rgb]{0.22352941,0.22352941,0.6745098}\makebox(0,0)[t]{\smash{\begin{tabular}[t]{c}$P_B$\end{tabular}}}}%
    \put(0.01855707,0.26235708){\color[rgb]{0.21960784,0.59607843,0.14509804}\makebox(0,0)[t]{\smash{\begin{tabular}[t]{c}{\footnotesize $U$}\end{tabular}}}}%
  \end{picture}%
\endgroup%

%% file: K_B-and-K_G-mod.pdf_tex
%% Creator: Inkscape 1.2 (dc2aeda, 2022-05-15), www.inkscape.org
%% Accompanies image file 'K_B-and-K_G-mod.pdf' (pdf, eps, ps)
\begingroup%
  \makeatletter%
  \providecommand\color[2][]{%
    \errmessage{(Inkscape) Color is used for the text in Inkscape, but the package 'color.sty' is not loaded}%
    \renewcommand\color[2][]{}%
  }%
  \providecommand\transparent[1]{%
    \errmessage{(Inkscape) Transparency is used (non-zero) for the text in Inkscape, but the package 'transparent.sty' is not loaded}%
    \renewcommand\transparent[1]{}%
  }%
  \providecommand\rotatebox[2]{#2}%
  \newcommand*\fsize{\dimexpr\f@size pt\relax}%
  \newcommand*\lineheight[1]{\fontsize{\fsize}{#1\fsize}\selectfont}%
  \ifx\svgwidth\undefined%
    \setlength{\unitlength}{854.99948096bp}%
    \ifx\svgscale\undefined%
      \relax%
    \else%
      \setlength{\unitlength}{\unitlength * \real{\svgscale}}%
    \fi%
  \else%
    \setlength{\unitlength}{\svgwidth}%
  \fi%
  \global\let\svgwidth\undefined%
  \global\let\svgscale\undefined%
  \makeatother%
  \begin{picture}(1,0.24815255)%
    \lineheight{1}%
    \setlength\tabcolsep{0pt}%
    \put(0,0){\includegraphics[width=\unitlength,page=1]{K_B-and-K_G-mod.pdf}}%
    \put(0.03424223,0.01868784){\color[rgb]{0.22352941,0.22352941,0.6745098}\makebox(0,0)[t]{\smash{\begin{tabular}[t]{c}$K_B$\end{tabular}}}}%
    \put(0,0){\includegraphics[width=\unitlength,page=2]{K_B-and-K_G-mod.pdf}}%
    \put(0.08706322,0.1736935){\color[rgb]{0.22352941,0.22352941,0.6745098}\makebox(0,0)[t]{\smash{\begin{tabular}[t]{c}{\footnotesize $P'_B$}\end{tabular}}}}%
    \put(0,0){\includegraphics[width=\unitlength,page=3]{K_B-and-K_G-mod.pdf}}%
    \put(0.57810211,0.01868784){\color[rgb]{0.21960784,0.59607843,0.14509804}\makebox(0,0)[t]{\smash{\begin{tabular}[t]{c}$K_G$\end{tabular}}}}%
    \put(0,0){\includegraphics[width=\unitlength,page=4]{K_B-and-K_G-mod.pdf}}%
    \put(0.6309231,0.1736935){\color[rgb]{0.21960784,0.59607843,0.14509804}\makebox(0,0)[t]{\smash{\begin{tabular}[t]{c}{\footnotesize $P'_G$}\end{tabular}}}}%
  \end{picture}%
\endgroup%

%% file: twisted-patterns.pdf_tex
%% Creator: Inkscape 1.2 (dc2aeda, 2022-05-15), www.inkscape.org
%% Accompanies image file 'twisted-patterns.pdf' (pdf, eps, ps)
\begingroup%
  \makeatletter%
  \providecommand\color[2][]{%
    \errmessage{(Inkscape) Color is used for the text in Inkscape, but the package 'color.sty' is not loaded}%
    \renewcommand\color[2][]{}%
  }%
  \providecommand\transparent[1]{%
    \errmessage{(Inkscape) Transparency is used (non-zero) for the text in Inkscape, but the package 'transparent.sty' is not loaded}%
    \renewcommand\transparent[1]{}%
  }%
  \providecommand\rotatebox[2]{#2}%
  \newcommand*\fsize{\dimexpr\f@size pt\relax}%
  \newcommand*\lineheight[1]{\fontsize{\fsize}{#1\fsize}\selectfont}%
  \ifx\svgwidth\undefined%
    \setlength{\unitlength}{1158.60017071bp}%
    \ifx\svgscale\undefined%
      \relax%
    \else%
      \setlength{\unitlength}{\unitlength * \real{\svgscale}}%
    \fi%
  \else%
    \setlength{\unitlength}{\svgwidth}%
  \fi%
  \global\let\svgwidth\undefined%
  \global\let\svgscale\undefined%
  \makeatother%
  \begin{picture}(1,0.37994109)%
    \lineheight{1}%
    \setlength\tabcolsep{0pt}%
    \put(0,0){\includegraphics[width=\unitlength,page=1]{twisted-patterns.pdf}}%
    \put(0.80352199,0.00720334){\color[rgb]{0.21960784,0.59607843,0.14509804}\makebox(0,0)[t]{\smash{\begin{tabular}[t]{c}$P'_{G}$\end{tabular}}}}%
    \put(0.21547873,0.00276885){\color[rgb]{0.22352941,0.22352941,0.6745098}\makebox(0,0)[t]{\smash{\begin{tabular}[t]{c}$P'_{B}$\end{tabular}}}}%
    \put(0,0){\includegraphics[width=\unitlength,page=2]{twisted-patterns.pdf}}%
    \put(0.66583194,0.20310018){\color[rgb]{0.21960784,0.59607843,0.14509804}\makebox(0,0)[t]{\smash{\begin{tabular}[t]{c}$qt$\end{tabular}}}}%
    \put(0.07471549,0.21638393){\color[rgb]{0.22352941,0.22352941,0.6745098}\makebox(0,0)[t]{\smash{\begin{tabular}[t]{c}$qt$\end{tabular}}}}%
  \end{picture}%
\endgroup%

%% file: pattern-twisting.pdf_tex
%% Creator: Inkscape 1.2 (dc2aeda, 2022-05-15), www.inkscape.org
%% Accompanies image file 'pattern-twisting.pdf' (pdf, eps, ps)
\begingroup%
  \makeatletter%
  \providecommand\color[2][]{%
    \errmessage{(Inkscape) Color is used for the text in Inkscape, but the package 'color.sty' is not loaded}%
    \renewcommand\color[2][]{}%
  }%
  \providecommand\transparent[1]{%
    \errmessage{(Inkscape) Transparency is used (non-zero) for the text in Inkscape, but the package 'transparent.sty' is not loaded}%
    \renewcommand\transparent[1]{}%
  }%
  \providecommand\rotatebox[2]{#2}%
  \newcommand*\fsize{\dimexpr\f@size pt\relax}%
  \newcommand*\lineheight[1]{\fontsize{\fsize}{#1\fsize}\selectfont}%
  \ifx\svgwidth\undefined%
    \setlength{\unitlength}{547.50046137bp}%
    \ifx\svgscale\undefined%
      \relax%
    \else%
      \setlength{\unitlength}{\unitlength * \real{\svgscale}}%
    \fi%
  \else%
    \setlength{\unitlength}{\svgwidth}%
  \fi%
  \global\let\svgwidth\undefined%
  \global\let\svgscale\undefined%
  \makeatother%
  \begin{picture}(1,0.18403176)%
    \lineheight{1}%
    \setlength\tabcolsep{0pt}%
    \put(0,0){\includegraphics[width=\unitlength,page=1]{pattern-twisting.pdf}}%
    \put(0.68493086,0.15757924){\color[rgb]{0.6,0.6,0.6}\makebox(0,0)[t]{\smash{\begin{tabular}[t]{c}$n$\end{tabular}}}}%
    \put(0,0){\includegraphics[width=\unitlength,page=2]{pattern-twisting.pdf}}%
    \put(0.68493086,-0.00353229){\color[rgb]{0.6,0.6,0.6}\makebox(0,0)[t]{\smash{\begin{tabular}[t]{c}$n$\end{tabular}}}}%
    \put(0,0){\includegraphics[width=\unitlength,page=3]{pattern-twisting.pdf}}%
    \put(0.85608696,0.16487083){\color[rgb]{0.6,0.6,0.6}\makebox(0,0)[t]{\smash{\begin{tabular}[t]{c}$n$\end{tabular}}}}%
    \put(0,0){\includegraphics[width=\unitlength,page=4]{pattern-twisting.pdf}}%
    \put(0.06833385,0.15757923){\color[rgb]{0.6,0.6,0.6}\makebox(0,0)[t]{\smash{\begin{tabular}[t]{c}$-2k$\end{tabular}}}}%
    \put(0,0){\includegraphics[width=\unitlength,page=5]{pattern-twisting.pdf}}%
    \put(0.11526542,0.0758303){\color[rgb]{0.6,0.6,0.6}\makebox(0,0)[t]{\smash{\begin{tabular}[t]{c}$-k$\end{tabular}}}}%
    \put(0,0){\includegraphics[width=\unitlength,page=6]{pattern-twisting.pdf}}%
    \put(0.26027379,0.0753873){\color[rgb]{0,0,0}\makebox(0,0)[t]{\smash{\begin{tabular}[t]{c}$n$\end{tabular}}}}%
    \put(0.36986271,0.0753873){\color[rgb]{0,0,0}\makebox(0,0)[t]{\smash{\begin{tabular}[t]{c}$=$\end{tabular}}}}%
    \put(0,0){\includegraphics[width=\unitlength,page=7]{pattern-twisting.pdf}}%
  \end{picture}%
\endgroup%

%% file: pos-neg-twisting.pdf_tex
%% Creator: Inkscape 1.2 (dc2aeda, 2022-05-15), www.inkscape.org
%% Accompanies image file 'pos-neg-twisting.pdf' (pdf, eps, ps)
\begingroup%
  \makeatletter%
  \providecommand\color[2][]{%
    \errmessage{(Inkscape) Color is used for the text in Inkscape, but the package 'color.sty' is not loaded}%
    \renewcommand\color[2][]{}%
  }%
  \providecommand\transparent[1]{%
    \errmessage{(Inkscape) Transparency is used (non-zero) for the text in Inkscape, but the package 'transparent.sty' is not loaded}%
    \renewcommand\transparent[1]{}%
  }%
  \providecommand\rotatebox[2]{#2}%
  \newcommand*\fsize{\dimexpr\f@size pt\relax}%
  \newcommand*\lineheight[1]{\fontsize{\fsize}{#1\fsize}\selectfont}%
  \ifx\svgwidth\undefined%
    \setlength{\unitlength}{1343.46053957bp}%
    \ifx\svgscale\undefined%
      \relax%
    \else%
      \setlength{\unitlength}{\unitlength * \real{\svgscale}}%
    \fi%
  \else%
    \setlength{\unitlength}{\svgwidth}%
  \fi%
  \global\let\svgwidth\undefined%
  \global\let\svgscale\undefined%
  \makeatother%
  \begin{picture}(1,0.3421028)%
    \lineheight{1}%
    \setlength\tabcolsep{0pt}%
    \put(0,0){\includegraphics[width=\unitlength,page=1]{pos-neg-twisting.pdf}}%
    \put(0.50003023,0.19097062){\color[rgb]{0,0,0}\makebox(0,0)[t]{\smash{\begin{tabular}[t]{c}(a)\end{tabular}}}}%
    \put(0,0){\includegraphics[width=\unitlength,page=2]{pos-neg-twisting.pdf}}%
    \put(0.20208637,0.00440341){\color[rgb]{0,0,0}\makebox(0,0)[t]{\smash{\begin{tabular}[t]{c}(b)\end{tabular}}}}%
    \put(0,0){\includegraphics[width=\unitlength,page=3]{pos-neg-twisting.pdf}}%
    \put(0.79797285,0.00440325){\color[rgb]{0,0,0}\makebox(0,0)[t]{\smash{\begin{tabular}[t]{c}(c)\end{tabular}}}}%
  \end{picture}%
\endgroup%

%% file: meridian-dance-part-3-mod.pdf_tex
%% Creator: Inkscape 1.2 (dc2aeda, 2022-05-15), www.inkscape.org
%% Accompanies image file 'meridian-dance-part-3-mod.pdf' (pdf, eps, ps)
\begingroup%
  \makeatletter%
  \providecommand\color[2][]{%
    \errmessage{(Inkscape) Color is used for the text in Inkscape, but the package 'color.sty' is not loaded}%
    \renewcommand\color[2][]{}%
  }%
  \providecommand\transparent[1]{%
    \errmessage{(Inkscape) Transparency is used (non-zero) for the text in Inkscape, but the package 'transparent.sty' is not loaded}%
    \renewcommand\transparent[1]{}%
  }%
  \providecommand\rotatebox[2]{#2}%
  \newcommand*\fsize{\dimexpr\f@size pt\relax}%
  \newcommand*\lineheight[1]{\fontsize{\fsize}{#1\fsize}\selectfont}%
  \ifx\svgwidth\undefined%
    \setlength{\unitlength}{1284.44942444bp}%
    \ifx\svgscale\undefined%
      \relax%
    \else%
      \setlength{\unitlength}{\unitlength * \real{\svgscale}}%
    \fi%
  \else%
    \setlength{\unitlength}{\svgwidth}%
  \fi%
  \global\let\svgwidth\undefined%
  \global\let\svgscale\undefined%
  \makeatother%
  \begin{picture}(1,0.79092271)%
    \lineheight{1}%
    \setlength\tabcolsep{0pt}%
    \put(0.2225497,0.55204908){\color[rgb]{0,0,0}\makebox(0,0)[t]{\smash{\begin{tabular}[t]{c}(a)\end{tabular}}}}%
    \put(0.78026495,0.55204908){\color[rgb]{0,0,0}\makebox(0,0)[t]{\smash{\begin{tabular}[t]{c}(b)\end{tabular}}}}%
    \put(0.2225497,0.25932293){\color[rgb]{0,0,0}\makebox(0,0)[t]{\smash{\begin{tabular}[t]{c}(c)\end{tabular}}}}%
    \put(0.78026495,0.25932293){\color[rgb]{0,0,0}\makebox(0,0)[t]{\smash{\begin{tabular}[t]{c}(d)\end{tabular}}}}%
    \put(0.15414887,0.65822037){\color[rgb]{0.22352941,0.22352941,0.6745098}\makebox(0,0)[t]{\smash{\begin{tabular}[t]{c}{\footnotesize $\mu_G$}\end{tabular}}}}%
    \put(0,0){\includegraphics[width=\unitlength,page=1]{meridian-dance-part-3-mod.pdf}}%
    \put(0.45582976,0.59628721){\color[rgb]{0.22352941,0.22352941,0.6745098}\makebox(0,0)[t]{\smash{\begin{tabular}[t]{c}{\footnotesize $\widehat B$}\end{tabular}}}}%
    \put(0,0){\includegraphics[width=\unitlength,page=2]{meridian-dance-part-3-mod.pdf}}%
    \put(0.09334933,0.61656522){\color[rgb]{0.50196078,0.50196078,0.50196078}\makebox(0,0)[t]{\smash{\begin{tabular}[t]{c}{\footnotesize $\beta$}\end{tabular}}}}%
    \put(0,0){\includegraphics[width=\unitlength,page=3]{meridian-dance-part-3-mod.pdf}}%
    \put(0.49793103,0.00249762){\color[rgb]{0,0,0}\makebox(0,0)[t]{\smash{\begin{tabular}[t]{c}(e)\end{tabular}}}}%
    \put(0.81175847,0.13966997){\color[rgb]{0.22352941,0.22352941,0.6745098}\makebox(0,0)[t]{\smash{\begin{tabular}[t]{c}{\scriptsize $\overline{C}_{t,q}(T_{r,s})$}\end{tabular}}}}%
    \put(0.903308,0.47337347){\color[rgb]{0.22352941,0.22352941,0.6745098}\makebox(0,0)[t]{\smash{\begin{tabular}[t]{c}{\scriptsize $\overline{C}_{t,q}(T_{r,s})$}\end{tabular}}}}%
    \put(0,0){\includegraphics[width=\unitlength,page=4]{meridian-dance-part-3-mod.pdf}}%
    \put(0.52210362,0.13842167){\color[rgb]{0.22352941,0.22352941,0.6745098}\makebox(0,0)[t]{\smash{\begin{tabular}[t]{c}{\scriptsize$(-t,-q)$}\end{tabular}}}}%
    \put(0,0){\includegraphics[width=\unitlength,page=5]{meridian-dance-part-3-mod.pdf}}%
    \put(0.77672401,0.05792539){\color[rgb]{0.22352941,0.22352941,0.6745098}\makebox(0,0)[t]{\smash{\begin{tabular}[t]{c}{\scriptsize $U$}\end{tabular}}}}%
  \end{picture}%
\endgroup%

%% file: meridian-dance-part-1-mod.pdf_tex
%% Creator: Inkscape 1.2 (dc2aeda, 2022-05-15), www.inkscape.org
%% Accompanies image file 'meridian-dance-part-1-mod.pdf' (pdf, eps, ps)
\begingroup%
  \makeatletter%
  \providecommand\color[2][]{%
    \errmessage{(Inkscape) Color is used for the text in Inkscape, but the package 'color.sty' is not loaded}%
    \renewcommand\color[2][]{}%
  }%
  \providecommand\transparent[1]{%
    \errmessage{(Inkscape) Transparency is used (non-zero) for the text in Inkscape, but the package 'transparent.sty' is not loaded}%
    \renewcommand\transparent[1]{}%
  }%
  \providecommand\rotatebox[2]{#2}%
  \newcommand*\fsize{\dimexpr\f@size pt\relax}%
  \newcommand*\lineheight[1]{\fontsize{\fsize}{#1\fsize}\selectfont}%
  \ifx\svgwidth\undefined%
    \setlength{\unitlength}{1662.1238312bp}%
    \ifx\svgscale\undefined%
      \relax%
    \else%
      \setlength{\unitlength}{\unitlength * \real{\svgscale}}%
    \fi%
  \else%
    \setlength{\unitlength}{\svgwidth}%
  \fi%
  \global\let\svgwidth\undefined%
  \global\let\svgscale\undefined%
  \makeatother%
  \begin{picture}(1,1.15026364)%
    \lineheight{1}%
    \setlength\tabcolsep{0pt}%
    \put(0,0){\includegraphics[width=\unitlength,page=1]{meridian-dance-part-1-mod.pdf}}%
    \put(0.30905064,0.61231853){\color[rgb]{0,0.50196078,0}\makebox(0,0)[t]{\smash{\begin{tabular}[t]{c}{\scriptsize$(s,-q)$}\end{tabular}}}}%
    \put(0,0){\includegraphics[width=\unitlength,page=2]{meridian-dance-part-1-mod.pdf}}%
    \put(0.86743154,0.60425102){\color[rgb]{0,0.50196078,0}\makebox(0,0)[t]{\smash{\begin{tabular}[t]{c}{\scriptsize$(s,-q)$}\end{tabular}}}}%
    \put(0,0){\includegraphics[width=\unitlength,page=3]{meridian-dance-part-1-mod.pdf}}%
    \put(0.22171331,0.52416653){\color[rgb]{0,0,0}\makebox(0,0)[t]{\smash{\begin{tabular}[t]{c}(c)\end{tabular}}}}%
    \put(0.7783133,0.52412989){\color[rgb]{0,0,0}\makebox(0,0)[t]{\smash{\begin{tabular}[t]{c}(d)\end{tabular}}}}%
    \put(0.22171331,0.87581192){\color[rgb]{0,0,0}\makebox(0,0)[t]{\smash{\begin{tabular}[t]{c}(a)\end{tabular}}}}%
    \put(0.7783133,0.87578028){\color[rgb]{0,0,0}\makebox(0,0)[t]{\smash{\begin{tabular}[t]{c}(b)\end{tabular}}}}%
    \put(0,0){\includegraphics[width=\unitlength,page=4]{meridian-dance-part-1-mod.pdf}}%
    \put(0.88900258,1.03275783){\color[rgb]{0,0.50196078,0}\makebox(0,0)[t]{\smash{\begin{tabular}[t]{c}{\scriptsize $C_{s,-q}(c')$}\end{tabular}}}}%
    \put(0.43721029,1.11168004){\color[rgb]{0.21960784,0.59607843,0.14509804}\makebox(0,0)[t]{\smash{\begin{tabular}[t]{c}{\footnotesize $\widehat G$}\end{tabular}}}}%
    \put(0,0){\includegraphics[width=\unitlength,page=5]{meridian-dance-part-1-mod.pdf}}%
    \put(0.15805592,1.05938778){\color[rgb]{0.50196078,0.50196078,0.50196078}\makebox(0,0)[t]{\smash{\begin{tabular}[t]{c}{\footnotesize $\beta$}\end{tabular}}}}%
    \put(0.13677711,0.89841894){\color[rgb]{0.21960784,0.59607843,0.14509804}\makebox(0,0)[t]{\smash{\begin{tabular}[t]{c}{\footnotesize $\mu_B$}\end{tabular}}}}%
    \put(0,0){\includegraphics[width=\unitlength,page=6]{meridian-dance-part-1-mod.pdf}}%
    \put(0.31083233,0.36758588){\color[rgb]{0,0.50196078,0}\makebox(0,0)[t]{\smash{\begin{tabular}[t]{c}{\scriptsize$(s,-q)$}\end{tabular}}}}%
    \put(0,0){\includegraphics[width=\unitlength,page=7]{meridian-dance-part-1-mod.pdf}}%
    \put(0.22171331,0.22703596){\color[rgb]{0,0,0}\makebox(0,0)[t]{\smash{\begin{tabular}[t]{c}(e)\end{tabular}}}}%
    \put(0.7783133,0.22699932){\color[rgb]{0,0,0}\makebox(0,0)[t]{\smash{\begin{tabular}[t]{c}(f)\end{tabular}}}}%
    \put(0.93717365,0.44348808){\color[rgb]{0.21960784,0.59607843,0.14509804}\makebox(0,0)[t]{\smash{\begin{tabular}[t]{c}{\scriptsize $\overline{C}_{t,q}(T_{r,s})$}\end{tabular}}}}%
    \put(0,0){\includegraphics[width=\unitlength,page=8]{meridian-dance-part-1-mod.pdf}}%
    \put(0.49459971,0.00193008){\color[rgb]{0,0,0}\makebox(0,0)[t]{\smash{\begin{tabular}[t]{c}(g)\end{tabular}}}}%
    \put(0.74661451,0.11447136){\color[rgb]{0.21960784,0.59607843,0.14509804}\makebox(0,0)[t]{\smash{\begin{tabular}[t]{c}{\scriptsize $\overline{C}_{t,q}(T_{r,s})$}\end{tabular}}}}%
    \put(0,0){\includegraphics[width=\unitlength,page=9]{meridian-dance-part-1-mod.pdf}}%
    \put(0.71502813,0.05017692){\color[rgb]{0.21960784,0.59607843,0.14509804}\makebox(0,0)[t]{\smash{\begin{tabular}[t]{c}{\scriptsize $U$}\end{tabular}}}}%
    \put(0,0){\includegraphics[width=\unitlength,page=10]{meridian-dance-part-1-mod.pdf}}%
    \put(0.51327807,0.11547987){\color[rgb]{0.21960784,0.59607843,0.14509804}\makebox(0,0)[t]{\smash{\begin{tabular}[t]{c}{\scriptsize$(-t,-q)$}\end{tabular}}}}%
    \put(0,0){\includegraphics[width=\unitlength,page=11]{meridian-dance-part-1-mod.pdf}}%
    \put(0.79727077,0.41261115){\color[rgb]{0.21960784,0.59607843,0.14509804}\makebox(0,0)[t]{\smash{\begin{tabular}[t]{c}{\scriptsize$(-t,-q)$}\end{tabular}}}}%
  \end{picture}%
\endgroup%

%% file: resolutions.pdf_tex
%% Creator: Inkscape 1.2 (dc2aeda, 2022-05-15), www.inkscape.org
%% Accompanies image file 'resolutions.pdf' (pdf, eps, ps)
\begingroup%
  \makeatletter%
  \providecommand\color[2][]{%
    \errmessage{(Inkscape) Color is used for the text in Inkscape, but the package 'color.sty' is not loaded}%
    \renewcommand\color[2][]{}%
  }%
  \providecommand\transparent[1]{%
    \errmessage{(Inkscape) Transparency is used (non-zero) for the text in Inkscape, but the package 'transparent.sty' is not loaded}%
    \renewcommand\transparent[1]{}%
  }%
  \providecommand\rotatebox[2]{#2}%
  \newcommand*\fsize{\dimexpr\f@size pt\relax}%
  \newcommand*\lineheight[1]{\fontsize{\fsize}{#1\fsize}\selectfont}%
  \ifx\svgwidth\undefined%
    \setlength{\unitlength}{752.2083673bp}%
    \ifx\svgscale\undefined%
      \relax%
    \else%
      \setlength{\unitlength}{\unitlength * \real{\svgscale}}%
    \fi%
  \else%
    \setlength{\unitlength}{\svgwidth}%
  \fi%
  \global\let\svgwidth\undefined%
  \global\let\svgscale\undefined%
  \makeatother%
  \begin{picture}(1,0.28990149)%
    \lineheight{1}%
    \setlength\tabcolsep{0pt}%
    \put(0,0){\includegraphics[width=\unitlength,page=1]{resolutions.pdf}}%
    \put(0.08691951,0.00698158){\color[rgb]{0,0,0}\makebox(0,0)[t]{\smash{\begin{tabular}[t]{c}$L_+$\end{tabular}}}}%
    \put(0.50016193,0.00698158){\color[rgb]{0,0,0}\makebox(0,0)[t]{\smash{\begin{tabular}[t]{c}$L_-$\end{tabular}}}}%
    \put(0,0){\includegraphics[width=\unitlength,page=2]{resolutions.pdf}}%
    \put(0.9131166,0.00698158){\color[rgb]{0,0,0}\makebox(0,0)[t]{\smash{\begin{tabular}[t]{c}$L_0$\end{tabular}}}}%
  \end{picture}%
\endgroup%

%% file: skein-schematic.pdf_tex
%% Creator: Inkscape 1.2 (dc2aeda, 2022-05-15), www.inkscape.org
%% Accompanies image file 'skein-schematic.pdf' (pdf, eps, ps)
\begingroup%
  \makeatletter%
  \providecommand\color[2][]{%
    \errmessage{(Inkscape) Color is used for the text in Inkscape, but the package 'color.sty' is not loaded}%
    \renewcommand\color[2][]{}%
  }%
  \providecommand\transparent[1]{%
    \errmessage{(Inkscape) Transparency is used (non-zero) for the text in Inkscape, but the package 'transparent.sty' is not loaded}%
    \renewcommand\transparent[1]{}%
  }%
  \providecommand\rotatebox[2]{#2}%
  \newcommand*\fsize{\dimexpr\f@size pt\relax}%
  \newcommand*\lineheight[1]{\fontsize{\fsize}{#1\fsize}\selectfont}%
  \ifx\svgwidth\undefined%
    \setlength{\unitlength}{1236.664174bp}%
    \ifx\svgscale\undefined%
      \relax%
    \else%
      \setlength{\unitlength}{\unitlength * \real{\svgscale}}%
    \fi%
  \else%
    \setlength{\unitlength}{\svgwidth}%
  \fi%
  \global\let\svgwidth\undefined%
  \global\let\svgscale\undefined%
  \makeatother%
  \begin{picture}(1,0.40479523)%
    \lineheight{1}%
    \setlength\tabcolsep{0pt}%
    \put(0,0){\includegraphics[width=\unitlength,page=1]{skein-schematic.pdf}}%
    \put(0.74791212,0.30978318){\color[rgb]{0,0,0}\makebox(0,0)[t]{\smash{\begin{tabular}[t]{c}$P$\end{tabular}}}}%
    \put(1.11405733,0.22552855){\color[rgb]{0,0,0}\makebox(0,0)[t]{\smash{\begin{tabular}[t]{c}$P(\overline{C}(R))$\end{tabular}}}}%
    \put(0.79498417,0.04234123){\color[rgb]{0.50196078,0.50196078,0.50196078}\makebox(0,0)[t]{\smash{\begin{tabular}[t]{c}$\beta$\end{tabular}}}}%
    \put(0,0){\includegraphics[width=\unitlength,page=2]{skein-schematic.pdf}}%
  \end{picture}%
\endgroup%

%% file: skein_double.pdf_tex
%% Creator: Inkscape 1.2 (dc2aeda, 2022-05-15), www.inkscape.org
%% Accompanies image file 'skein_double.pdf' (pdf, eps, ps)
\begingroup%
  \makeatletter%
  \providecommand\color[2][]{%
    \errmessage{(Inkscape) Color is used for the text in Inkscape, but the package 'color.sty' is not loaded}%
    \renewcommand\color[2][]{}%
  }%
  \providecommand\transparent[1]{%
    \errmessage{(Inkscape) Transparency is used (non-zero) for the text in Inkscape, but the package 'transparent.sty' is not loaded}%
    \renewcommand\transparent[1]{}%
  }%
  \providecommand\rotatebox[2]{#2}%
  \newcommand*\fsize{\dimexpr\f@size pt\relax}%
  \newcommand*\lineheight[1]{\fontsize{\fsize}{#1\fsize}\selectfont}%
  \ifx\svgwidth\undefined%
    \setlength{\unitlength}{1437.81943013bp}%
    \ifx\svgscale\undefined%
      \relax%
    \else%
      \setlength{\unitlength}{\unitlength * \real{\svgscale}}%
    \fi%
  \else%
    \setlength{\unitlength}{\svgwidth}%
  \fi%
  \global\let\svgwidth\undefined%
  \global\let\svgscale\undefined%
  \makeatother%
  \begin{picture}(1,0.40363392)%
    \lineheight{1}%
    \setlength\tabcolsep{0pt}%
    \put(0,0){\includegraphics[width=\unitlength,page=1]{skein_double.pdf}}%
    \put(0.46558953,0.22161581){\color[rgb]{0,0,0}\makebox(0,0)[t]{\smash{\begin{tabular}[t]{c}$L_+$\end{tabular}}}}%
    \put(0,0){\includegraphics[width=\unitlength,page=2]{skein_double.pdf}}%
    \put(0.1493045,0.00365257){\color[rgb]{0,0,0}\makebox(0,0)[t]{\smash{\begin{tabular}[t]{c}$L_-$\end{tabular}}}}%
    \put(0,0){\includegraphics[width=\unitlength,page=3]{skein_double.pdf}}%
    \put(0.74014471,0.00365257){\color[rgb]{0,0,0}\makebox(0,0)[t]{\smash{\begin{tabular}[t]{c}$L_0$\end{tabular}}}}%
    \put(0,0){\includegraphics[width=\unitlength,page=4]{skein_double.pdf}}%
  \end{picture}%
\endgroup%

%% file: twisted-patterns-skein.pdf_tex
%% Creator: Inkscape 1.2 (dc2aeda, 2022-05-15), www.inkscape.org
%% Accompanies image file 'twisted-patterns-skein.pdf' (pdf, eps, ps)
\begingroup%
  \makeatletter%
  \providecommand\color[2][]{%
    \errmessage{(Inkscape) Color is used for the text in Inkscape, but the package 'color.sty' is not loaded}%
    \renewcommand\color[2][]{}%
  }%
  \providecommand\transparent[1]{%
    \errmessage{(Inkscape) Transparency is used (non-zero) for the text in Inkscape, but the package 'transparent.sty' is not loaded}%
    \renewcommand\transparent[1]{}%
  }%
  \providecommand\rotatebox[2]{#2}%
  \newcommand*\fsize{\dimexpr\f@size pt\relax}%
  \newcommand*\lineheight[1]{\fontsize{\fsize}{#1\fsize}\selectfont}%
  \ifx\svgwidth\undefined%
    \setlength{\unitlength}{1659.03883566bp}%
    \ifx\svgscale\undefined%
      \relax%
    \else%
      \setlength{\unitlength}{\unitlength * \real{\svgscale}}%
    \fi%
  \else%
    \setlength{\unitlength}{\svgwidth}%
  \fi%
  \global\let\svgwidth\undefined%
  \global\let\svgscale\undefined%
  \makeatother%
  \begin{picture}(1,0.25261852)%
    \lineheight{1}%
    \setlength\tabcolsep{0pt}%
    \put(0,0){\includegraphics[width=\unitlength,page=1]{twisted-patterns-skein.pdf}}%
    \put(0.14729686,0.00135621){\color[rgb]{0,0,0}\makebox(0,0)[t]{\smash{\begin{tabular}[t]{c}$L_+$\end{tabular}}}}%
    \put(0,0){\includegraphics[width=\unitlength,page=2]{twisted-patterns-skein.pdf}}%
    \put(0.0501849,0.13816201){\color[rgb]{0,0,0}\makebox(0,0)[t]{\smash{\begin{tabular}[t]{c}$m$\end{tabular}}}}%
    \put(0,0){\includegraphics[width=\unitlength,page=3]{twisted-patterns-skein.pdf}}%
    \put(0.27351104,0.13816201){\color[rgb]{0,0,0}\makebox(0,0)[t]{\smash{\begin{tabular}[t]{c}$n$\end{tabular}}}}%
    \put(0.49991422,0.00135621){\color[rgb]{0,0,0}\makebox(0,0)[t]{\smash{\begin{tabular}[t]{c}$L_-$\end{tabular}}}}%
    \put(0,0){\includegraphics[width=\unitlength,page=4]{twisted-patterns-skein.pdf}}%
    \put(0.4028056,0.13816201){\color[rgb]{0,0,0}\makebox(0,0)[t]{\smash{\begin{tabular}[t]{c}$m$\end{tabular}}}}%
    \put(0,0){\includegraphics[width=\unitlength,page=5]{twisted-patterns-skein.pdf}}%
    \put(0.62613173,0.13816201){\color[rgb]{0,0,0}\makebox(0,0)[t]{\smash{\begin{tabular}[t]{c}$n$\end{tabular}}}}%
    \put(0.85252865,0.00135621){\color[rgb]{0,0,0}\makebox(0,0)[t]{\smash{\begin{tabular}[t]{c}$L_0$\end{tabular}}}}%
    \put(0,0){\includegraphics[width=\unitlength,page=6]{twisted-patterns-skein.pdf}}%
    \put(0.75542295,0.13816201){\color[rgb]{0,0,0}\makebox(0,0)[t]{\smash{\begin{tabular}[t]{c}$m$\end{tabular}}}}%
    \put(0,0){\includegraphics[width=\unitlength,page=7]{twisted-patterns-skein.pdf}}%
    \put(0.97874908,0.13816201){\color[rgb]{0,0,0}\makebox(0,0)[t]{\smash{\begin{tabular}[t]{c}$n$\end{tabular}}}}%
    \put(0,0){\includegraphics[width=\unitlength,page=8]{twisted-patterns-skein.pdf}}%
  \end{picture}%
\endgroup%

%% file: big-skein.pdf_tex
%% Creator: Inkscape 1.2 (dc2aeda, 2022-05-15), www.inkscape.org
%% Accompanies image file 'big-skein.pdf' (pdf, eps, ps)
\begingroup%
  \makeatletter%
  \providecommand\color[2][]{%
    \errmessage{(Inkscape) Color is used for the text in Inkscape, but the package 'color.sty' is not loaded}%
    \renewcommand\color[2][]{}%
  }%
  \providecommand\transparent[1]{%
    \errmessage{(Inkscape) Transparency is used (non-zero) for the text in Inkscape, but the package 'transparent.sty' is not loaded}%
    \renewcommand\transparent[1]{}%
  }%
  \providecommand\rotatebox[2]{#2}%
  \newcommand*\fsize{\dimexpr\f@size pt\relax}%
  \newcommand*\lineheight[1]{\fontsize{\fsize}{#1\fsize}\selectfont}%
  \ifx\svgwidth\undefined%
    \setlength{\unitlength}{2296.64933332bp}%
    \ifx\svgscale\undefined%
      \relax%
    \else%
      \setlength{\unitlength}{\unitlength * \real{\svgscale}}%
    \fi%
  \else%
    \setlength{\unitlength}{\svgwidth}%
  \fi%
  \global\let\svgwidth\undefined%
  \global\let\svgscale\undefined%
  \makeatother%
  \begin{picture}(1,1.22228655)%
    \lineheight{1}%
    \setlength\tabcolsep{0pt}%
    \put(0,0){\includegraphics[width=\unitlength,page=1]{big-skein.pdf}}%
    \put(0.92408023,0.72764715){\color[rgb]{0,0,0}\makebox(0,0)[t]{\smash{\begin{tabular}[t]{c}Figure~\ref{figure:big_skein_2}\end{tabular}}}}%
    \put(0.10595506,1.00296146){\color[rgb]{0,0,0}\makebox(0,0)[t]{\smash{\begin{tabular}[t]{c}$K_B=[D^1 \circ D^2_{qt}]$\end{tabular}}}}%
    \put(0.58548468,1.00655724){\color[rgb]{0,0,0}\makebox(0,0)[t]{\smash{\begin{tabular}[t]{c}$U$\end{tabular}}}}%
    \put(0.58200732,0.23472382){\color[rgb]{0,0,0}\makebox(0,0)[t]{\smash{\begin{tabular}[t]{c}$U$\end{tabular}}}}%
    \put(0.50636158,1.13902878){\color[rgb]{0,0,0}\makebox(0,0)[t]{\smash{\begin{tabular}[t]{c}${-}$\end{tabular}}}}%
    \put(0.50491609,0.57235491){\color[rgb]{0,0,0}\makebox(0,0)[t]{\smash{\begin{tabular}[t]{c}${-}$\end{tabular}}}}%
    \put(0.75966814,0.44348037){\color[rgb]{0,0,0}\makebox(0,0)[t]{\smash{\begin{tabular}[t]{c}${-}$\end{tabular}}}}%
    \put(0.21564459,0.97299075){\color[rgb]{0,0,0}\makebox(0,0)[t]{\smash{\begin{tabular}[t]{c}${0}$\end{tabular}}}}%
    \put(0.50156634,0.83931962){\color[rgb]{0,0,0}\makebox(0,0)[t]{\smash{\begin{tabular}[t]{c}${\cup}$\end{tabular}}}}%
    \put(0.21025137,0.67807818){\color[rgb]{0,0,0}\makebox(0,0)[t]{\smash{\begin{tabular}[t]{c}${\cup}$\end{tabular}}}}%
    \put(0.20785352,0.20633661){\color[rgb]{0,0,0}\makebox(0,0)[t]{\smash{\begin{tabular}[t]{c}${\cup}$\end{tabular}}}}%
    \put(0.48598285,0.20753591){\color[rgb]{0,0,0}\makebox(0,0)[t]{\smash{\begin{tabular}[t]{c}${\cup}$\end{tabular}}}}%
    \put(0.06999027,0.70624632){\color[rgb]{0,0,0}\makebox(0,0)[t]{\smash{\begin{tabular}[t]{c}$[\overline{C}_{0,2} \circ D^2_{qt}]$\end{tabular}}}}%
    \put(0.62144947,0.74700619){\color[rgb]{0,0,0}\makebox(0,0)[t]{\smash{\begin{tabular}[t]{c}$D^2_{qt}(C(R)) $\end{tabular}}}}%
    \put(0.21564459,0.44249954){\color[rgb]{0,0,0}\makebox(0,0)[t]{\smash{\begin{tabular}[t]{c}${0}$\end{tabular}}}}%
    \put(0.499768,0.44969231){\color[rgb]{0,0,0}\makebox(0,0)[t]{\smash{\begin{tabular}[t]{c}${0}$\end{tabular}}}}%
    \put(0.04601379,0.46407726){\color[rgb]{0,0,0}\makebox(0,0)[t]{\smash{\begin{tabular}[t]{c}$[ D^2_{qt}]$\end{tabular}}}}%
    \put(0.59747706,0.46407726){\color[rgb]{0,0,0}\makebox(0,0)[t]{\smash{\begin{tabular}[t]{c}$[ D^1_{qt}]$\end{tabular}}}}%
    \put(0.0552737,0.23030118){\color[rgb]{0,0,0}\makebox(0,0)[t]{\smash{\begin{tabular}[t]{c}$[\overline{C}_{2qt,2}]$\end{tabular}}}}%
    \put(0.59279693,0.02672421){\color[rgb]{0,0,0}\makebox(0,0)[t]{\smash{\begin{tabular}[t]{c}$C(R)$\end{tabular}}}}%
    \put(0.03402563,0.00252099){\color[rgb]{0,0,0}\makebox(0,0)[t]{\smash{\begin{tabular}[t]{c}$[H]$\end{tabular}}}}%
    \put(0,0){\includegraphics[width=\unitlength,page=2]{big-skein.pdf}}%
  \end{picture}%
\endgroup%

%% file: big-skein-2.pdf_tex
%% Creator: Inkscape 1.2 (dc2aeda, 2022-05-15), www.inkscape.org
%% Accompanies image file 'big-skein-2.pdf' (pdf, eps, ps)
\begingroup%
  \makeatletter%
  \providecommand\color[2][]{%
    \errmessage{(Inkscape) Color is used for the text in Inkscape, but the package 'color.sty' is not loaded}%
    \renewcommand\color[2][]{}%
  }%
  \providecommand\transparent[1]{%
    \errmessage{(Inkscape) Transparency is used (non-zero) for the text in Inkscape, but the package 'transparent.sty' is not loaded}%
    \renewcommand\transparent[1]{}%
  }%
  \providecommand\rotatebox[2]{#2}%
  \newcommand*\fsize{\dimexpr\f@size pt\relax}%
  \newcommand*\lineheight[1]{\fontsize{\fsize}{#1\fsize}\selectfont}%
  \ifx\svgwidth\undefined%
    \setlength{\unitlength}{1399.25482322bp}%
    \ifx\svgscale\undefined%
      \relax%
    \else%
      \setlength{\unitlength}{\unitlength * \real{\svgscale}}%
    \fi%
  \else%
    \setlength{\unitlength}{\svgwidth}%
  \fi%
  \global\let\svgwidth\undefined%
  \global\let\svgscale\undefined%
  \makeatother%
  \begin{picture}(1,0.975922)%
    \lineheight{1}%
    \setlength\tabcolsep{0pt}%
    \put(0,0){\includegraphics[width=\unitlength,page=1]{big-skein-2.pdf}}%
    \put(0.5382436,0.53954847){\color[rgb]{0,0,0}\makebox(0,0)[t]{\smash{\begin{tabular}[t]{c}${0}$\end{tabular}}}}%
    \put(0,0){\includegraphics[width=\unitlength,page=2]{big-skein-2.pdf}}%
    \put(0.19142676,0.32151403){\color[rgb]{0,0,0}\makebox(0,0)[t]{\smash{\begin{tabular}[t]{c}${-}$\end{tabular}}}}%
    \put(0.43406004,0.71104539){\color[rgb]{0,0,0}\makebox(0,0)[t]{\smash{\begin{tabular}[t]{c}${-}$\end{tabular}}}}%
    \put(0.25862367,0.72603734){\color[rgb]{0,0,0}\makebox(0,0)[t]{\smash{\begin{tabular}[t]{c}$D^2_{qt}(C(R)) $\end{tabular}}}}%
    \put(0.08297904,0.37797842){\color[rgb]{0,0,0}\makebox(0,0)[t]{\smash{\begin{tabular}[t]{c}$D^1_{qt}(C(R)) $\end{tabular}}}}%
    \put(0,0){\includegraphics[width=\unitlength,page=3]{big-skein-2.pdf}}%
    \put(0.64185942,0.71104539){\color[rgb]{0,0,0}\makebox(0,0)[t]{\smash{\begin{tabular}[t]{c}${0}$\end{tabular}}}}%
    \put(0.03340387,0.00238798){\color[rgb]{0,0,0}\makebox(0,0)[t]{\smash{\begin{tabular}[t]{c}$U$\end{tabular}}}}%
    \put(0,0){\includegraphics[width=\unitlength,page=4]{big-skein-2.pdf}}%
    \put(0.74969572,0.29090898){\color[rgb]{0,0,0}\makebox(0,0)[t]{\smash{\begin{tabular}[t]{c}${\cup}$\end{tabular}}}}%
    \put(0.87610799,0.29090898){\color[rgb]{0,0,0}\makebox(0,0)[t]{\smash{\begin{tabular}[t]{c}${\cup}$\end{tabular}}}}%
    \put(0.67969302,0.00691067){\color[rgb]{0,0,0}\makebox(0,0)[t]{\smash{\begin{tabular}[t]{c}$C(R)$\end{tabular}}}}%
    \put(0.7544614,0.37892094){\color[rgb]{0,0,0}\makebox(0,0)[t]{\smash{\begin{tabular}[t]{c}$\overline{C}_{2qt,2}(C(R))$\end{tabular}}}}%
    \put(0,0){\includegraphics[width=\unitlength,page=5]{big-skein-2.pdf}}%
  \end{picture}%
\endgroup%

%% file: skein_tree_K_B.tex
\begin{tikzpicture}
    \begin{scope}
        \node (K_B) at (0,0) {$[D^1 \circ D^2_{qt}]$};
            \node (L) at (-2,-2) {$U$};
            \node (R) at (2,-2) {$[\overline{C}_{0,2} \circ D^2_{qt}]$};
                \node (Rl) at (-2,-3) {$[D^2_{qt}]$};
                    \node (RlL) at (-4,-5) {$[D^1_{qt}]$};
                        \node (RlLL) at (-6,-7) {$U$};
                        \node (RlLR) at (-2,-7) {$[\overline{C}_{2qt,2}]$}; 
                            \node (RlLRl) at (-4,-8) {$[H]$};
                            \node (RlLRr) at (0,-8) {$C(R)$};
                    \node (RlR) at (0,-5) {$[\overline{C}_{2qt,2}]$};
                        \node (RlRl) at (-2,-6) {$[H]$};
                        \node (RlRr) at (2,-6) {$C(R)$};
                \node (Rr) at (6,-3) {$D^2_{qt}(C(R))$};
                    \node (RrL) at (4,-5) {$D^1_{qt}(C(R))$};
                        \node (RrLL) at (2,-7) {$U$};
                        \node (RrLR) at (6,-7) {$\overline{C}_{2qt,2}(C(R))$}; 
                            \node (RrLRl) at (4,-8) {$\overline{C}(R)$};
                            \node (RrLRr) at (8,-8) {$C(R)$};
                    \node (RrR) at (8,-5) {$\overline{C}_{2qt,2}(C(R))$}; 
                        \node (RrRl) at (6,-6) {$\overline{C}(R)$};
                        \node (RrRr) at (10,-6) {$C(R)$};
    \end{scope}
    
    \begin{scope}
        \path [->] (K_B) edge node [anchor=east] {$-$} (L);
        \path [->] (K_B) edge node [anchor=west] {$0$} (R); 
        
            \path [->] (Rl) edge node [anchor=east] {$-$} (RlL);
            \path [->] (Rl) edge node [anchor=west] {$0$} (RlR);

                \path [->] (RlL) edge node [anchor=east] {$-$} (RlLL);
                \path [->] (RlL) edge node [anchor=west] {$0$} (RlLR);
        
            \path [->] (Rr) edge node [anchor=east] {$-$} (RrL);
            \path [->] (Rr) edge node [anchor=west] {$0$} (RrR);

                \path [->] (RrL) edge node [anchor=east] {$-$} (RrLL);
                \path [->] (RrL) edge node [anchor=west] {$0$} (RrLR);
    \end{scope}
    
    \begin{scope}
        \node [draw] (R-split) at (2,-3) {$0$};
        \path [dashed] (R) edge (R-split);
        \path [->, dashed] (R-split) edge (Rl);
        \path [->, dashed] (R-split) edge (Rr);

        \node [draw] (RlR-split) at (0,-6) {$q(r-t)$};
        \path [dashed] (RlR) edge (RlR-split);
        \path [->, dashed] (RlR-split) edge (RlRl);
        \path [->, dashed] (RlR-split) edge (RlRr);

        \node [draw] (RlLR-split) at (-2,-8) {$q(r-t)$};
        \path [dashed] (RlLR) edge (RlLR-split);
        \path [->, dashed] (RlLR-split) edge (RlLRl);
        \path [->, dashed] (RlLR-split) edge (RlLRr);   

        \node [draw] (RrLR-split) at (6,-8) {$-qt$};
        \path [dashed] (RrLR) edge (RrLR-split);
        \path [->, dashed] (RrLR-split) edge (RrLRl);
        \path [->, dashed] (RrLR-split) edge (RrLRr);    

        \node [draw] (RrR-split) at (8,-6) {$-qt$};
        \path [dashed] (RrR) edge (RrR-split);
        \path [->, dashed] (RrR-split) edge (RrRl);
        \path [->, dashed] (RrR-split) edge (RrRr);    
    \end{scope} 

    \begin{scope}
        \node (lemma_K_B) at (0,-1.5) {Lemma \ref{lemma:double_double}}; 
            \node (lemma_R) at (2,-3.5) {Lemma \ref{lemma:cable_double}};
                \node (lemma_Rl) at (-2,-4.5) {Lemma \ref{lemma:double}};
                    \node (lemma:RlL) at (-4,-6.5) {Lemma \ref{lemma:double}};
                        \node (lemma:RlLR) at (-2,-8.5) {Lemma \ref{lemma:cable}};
                    \node (lemma_RlR) at (0,-6.5) {Lemma \ref{lemma:cable}};
                \node (lemma:Rr) at (6,-4.5) {Lemma \ref{lemma:double_cable}}; 
                    \node (lemma:RrL) at (4,-6.5) {Lemma \ref{lemma:double_cable}}; 
                            \node (lemma:RrLR) at (6,-8.5) {Lemma \ref{lemma:cable_cable}};
                        \node (lemma:RrR) at (8,-6.5) {Lemma \ref{lemma:cable_cable}};
    \end{scope}
\end{tikzpicture}

%% file: skein_tree_K_G.tex
\begin{tikzpicture}
    \begin{scope}
        \node (K_G) at (0,0) {$[D^2 \circ D^1_{qt}]$};
            \node (L) at (-5,-5) {$[D^1 \circ D^1_{qt}]$};
                \node (LL) at (-7,-7) {$U$};
                \node (LR) at (-3,-7) {$[\overline{C}_{0,2} \circ D^1_{qt}]$};
                    \node (LRl) at (-6,-8) {$[D^1_{qt}]$};
                        \node (LRlL) at (-8,-10) {$U$};
                        \node (LRlR) at (-4,-10) {$[\overline{C}_{2qt,2}]$};
                            \node (LRlRl) at (-6,-11) {$[H]$};
                            \node (LRlRr) at (-2,-11) {$C(R)$};
                    \node (LRr) at (0,-8) {$D^1_{qt}(C(R))$};
                        \node (LRrL) at (-2,-10) {$U$};
                        \node (LRrR) at (2,-10) {$\overline{C}_{2qt,2}(C(R))$};
                            \node (LRrRl) at (0,-11) {$\overline{C}(R)$};
                            \node (LRrRr) at (4,-11) {$C(R)$};
            \node (R) at (2,-2) {$[\overline{C}_{0,2} \circ D^1_{qt}]$};
                \node (Rl) at (-1,-3) {$[D^1_{qt}]$};
                    \node (RlL) at (-3,-5) {$U$};
                    \node (RlR) at (1,-5) {$[\overline{C}_{2qt,2}]$};
                        \node (RlRl) at (-1,-6) {$[H]$};
                        \node (RlRr) at (3,-6) {$C(R)$};
                \node (Rr) at (5,-3) {$D^1_{qt}(C(R))$};
                    \node (RrL) at (3,-5) {$U$};
                    \node (RrR) at (7,-5) {$\overline{C}_{2qt,2}(C(R))$};
                        \node (RrRl) at (5,-6) {$\overline{C}(R)$};
                        \node (RrRr) at (9,-6) {$C(R)$};
    \end{scope}
    
    \begin{scope}
        \path [->] (K_G) edge node [anchor=east] {$-$} (L);
        \path [->] (K_G) edge node [anchor=west] {$0$} (R); 
        
            \path [->] (L) edge node [anchor=east] {$-$} (LL);
            \path [->] (L) edge node [anchor=west] {$0$} (LR);

                \path [->] (LRl) edge node [anchor=east] {$-$} (LRlL);
                \path [->] (LRl) edge node [anchor=west] {$0$} (LRlR);
        
                \path [->] (LRr) edge node [anchor=east] {$-$} (LRrL);
                \path [->] (LRr) edge node [anchor=west] {$0$} (LRrR);

            \path [->] (Rl) edge node [anchor=east] {$-$} (RlL);
            \path [->] (Rl) edge node [anchor=west] {$0$} (RlR);
    
            \path [->] (Rr) edge node [anchor=east] {$-$} (RrL);
            \path [->] (Rr) edge node [anchor=west] {$0$} (RrR);
    \end{scope}   
    
    \begin{scope}
        \node [draw] (LR-split) at (-3,-8) {$0$};
        \path [dashed] (LR) edge (LR-split);
        \path [->, dashed] (LR-split) edge (LRl);
        \path [->, dashed] (LR-split) edge (LRr);

            \node [draw] (LRlR-split) at (-4,-11) {$q(r-t)$};
            \path [dashed] (LRlR) edge (LRlR-split);
            \path [->, dashed] (LRlR-split) edge (LRlRl);
            \path [->, dashed] (LRlR-split) edge (LRlRr);
    
            \node [draw] (LRrR-split) at (2,-11) {$-qt$};
            \path [dashed] (LRrR) edge (LRrR-split);
            \path [->, dashed] (LRrR-split) edge (LRrRl);
            \path [->, dashed] (LRrR-split) edge (LRrRr); 

        \node [draw] (R-split) at (2,-3) {$0$};
        \path [dashed] (R) edge (R-split);
        \path [->, dashed] (R-split) edge (Rl);
        \path [->, dashed] (R-split) edge (Rr); 

            \node [draw] (RlR-split) at (1,-6) {$q(r-t)$};
            \path [dashed] (RlR) edge (RlR-split);
            \path [->, dashed] (RlR-split) edge (RlRl);
            \path [->, dashed] (RlR-split) edge (RlRr); 
    
            \node [draw] (RrR-split) at (7,-6) {$-qt$};
            \path [dashed] (RrR) edge (RrR-split);
            \path [->, dashed] (RrR-split) edge (RrRl);
            \path [->, dashed] (RrR-split) edge (RrRr); 
    \end{scope} 

    \begin{scope}
        \node (lemma_K_G) at (0,-1.5) {Lemma \ref{lemma:double_double}}; 
            \node (lemma_L) at (-5,-6.5) {Lemma \ref{lemma:double_double}};
                \node (lemma_LR) at (-3,-8.5) {Lemma \ref{lemma:cable_double}};
                    \node (lemma_LRl) at (-6,-9.5) {Lemma \ref{lemma:double}};
                        \node (lemma_LRlR) at (-4,-11.5) {Lemma \ref{lemma:cable}};
                    \node (lemma_LRr) at (0,-9.5) {Lemma \ref{lemma:double_cable}};
                        \node (lemma_LRrR) at (2,-11.5) {Lemma \ref{lemma:cable_cable}};
            \node (lemma_R) at (2,-3.5) {Lemma \ref{lemma:cable_double}};
                \node (lemma_Rl) at (-1,-4.5) {Lemma \ref{lemma:double}};
                    \node (lemma_RlR) at (1,-6.5) {Lemma \ref{lemma:cable}};
                \node (lemma_Rr) at (5,-4.5) {Lemma \ref{lemma:double_cable}};
                    \node (lemma_RrR) at (7,-6.5) {Lemma \ref{lemma:cable_cable}};
    \end{scope}
\end{tikzpicture}